\documentclass{amsart}
\usepackage{amssymb}
\usepackage{braket}
\usepackage{mathrsfs}
\usepackage[all]{xy}
\usepackage{graphicx}

\newtheorem{thm}{Theorem}[section]
\newtheorem{cor}[thm]{Corollary}
\newtheorem{prop}[thm]{Proposition}
\theoremstyle{definition}
\newtheorem{dfn}[thm]{Definition}
\newtheorem{ex}[thm]{Example}
\newtheorem{claim}[thm]{Claim}
\newtheorem{lem}[thm]{Lemma}
\newtheorem{fact}[thm]{Fact}

\theoremstyle{remark}
\newtheorem{rem}[thm]{Remark}

\newcommand{\Ob}{\mathrm{Ob}}
\newcommand{\id}{\mathrm{id}}
\newcommand{\ppr}{^{\prime}}
\newcommand{\Map}{\mathrm{Map}}

\newcommand{\Mon}{\mathit{Mon}}
\newcommand{\Sett}{\mathit{Set}}

\newcommand{\Ring}{\mathit{Ring}}

\newcommand{\TamG}{\mathit{Tam}(G)}
\newcommand{\Gs}{{}_G\mathit{set}}

\newcommand{\Z}{\mathbb{Z}}
\newcommand{\N}{\mathbb{N}}
\newcommand{\pro}{\mathrm{pr}}
\newcommand{\Pro}{\mathrm{Pr}}
\newcommand{\res}{\mathrm{res}}
\newcommand{\ind}{\mathrm{ind}}
\newcommand{\jnd}{\mathrm{jnd}}
\newcommand{\I}{\mathscr{I}}
\newcommand{\J}{\mathscr{J}}
\newcommand{\Ical}{\mathcal{I}}

\newcommand{\Lcal}{\mathcal{L}}
\newcommand{\Pcal}{\mathcal{P}}
\newcommand{\Qcal}{\mathcal{Q}}
\newcommand{\Scal}{\mathcal{S}}
\newcommand{\Spec}{\mathit{Spec}}

\numberwithin{equation}{section}

\begin{document}

\title[Spectrum of the Burnside Tambara functor on a cyclic $p$-group]{Spectrum of the Burnside Tambara functor on a cyclic $p$-group}

\author{Hiroyuki NAKAOKA}
\address{Department of Mathematics and Computer Science, Kagoshima University, 1-21-35 Korimoto, Kagoshima, 890-0065 Japan}

\email{nakaoka@sci.kagoshima-u.ac.jp}


\thanks{The author wishes to thank Professor Fumihito Oda for his suggestions and useful comments.}
\thanks{The author also wishes to thank Professor Akihiko Hida for his question and comments.}
\thanks{This work is supported by JSPS Grant-in-Aid for Young Scientists (B) 22740005}

\begin{abstract}
For a finite group $G$, a Tambara functor on $G$ is regarded as a $G$-bivariant analog of a commutative ring. In this analogy, previously we have defined an ideal of a Tambara functor. In this article, we will demonstrate a calculation of the prime spectrum of the Burnside Tambara functor, when $G$ is a cyclic $p$-group for a prime integer $p$.
\end{abstract}

\maketitle

\section{Introduction and Preliminaries}

A {\it Tambara functor} on a finite group $G$ is firstly defined by Tambara \cite{Tam} in the name \lq TNR-functor', to treat the multiplicative transfers of Green functors. The terminology {\it Tambara functor} firstly appeared in Brun's paper, when he used it to describe the structure of Witt-Burnside rings \cite{Brun}. It consists of a triplet $T=(T^{\ast},T_+,T_{\bullet})$, where the {\it additive part} $(T^{\ast},T_+)$ forms a Mackey functor, whereas the {\it multiplicative part} $(T^{\ast},T_{\bullet})$ forms a semi-Mackey functor. This is just like a commutative ring consists of an additive abelian group structure and a multiplicative commutative semi-group structure. In fact a Tambara functor is nothing other than a commutative ring when $G$ is trivial. In this sense, this notion is regarded as a Mackey-functor-theoretic analog (or, \lq $G$-bivariant analog') of a commutative ring \cite{Yoshida}.

In this analogy, some algebraic notions in commutative ring theory find their analogs in Tambara functor theory. We have ideals \cite{N_Ideal}, fractions \cite{N_Frac}, and polynomials \cite{N_TamMack} of Tambara functors. These are mutually related, as they should be, and moreover in connection with the celebrated Dress construction \cite{O-Y3}, \cite{N_DressPolyHopf}. The most typical Tambara functor is the {\it Burnside Tambara functor} $\Omega_G$ (Example \ref{ExTam}), which plays a role just like $\Z$ in the ordinary commutative ring theory.

Especially, in the $G$-bivariant analog of the ideal theory, the prime spectrum $\Spec\,T$ has been defined for any Tambara functor $T$. In our previous short note \cite{N_Specp}, we have demonstrated a calculation of the prime spectrum of $\Omega_G$, when $G$ is a group of prime order $p$. In this article, extending the method in \cite{N_Specp}, we calculate $\Spec\,\Omega_G$ when $G$ is a {\it cyclic $p$-group} for a prime integer $p$. 
To determine the prime ideals of $\Omega_G$, the key observation is that the set of subgroups of $G$ is totally ordered
\[ e=H_0<H_1<\cdots<H_r=G, \]
and thus $\Omega_G$ can be regarded as a sequence of commutative rings equipped with adjacent structure morphisms
\[
\xy
(-48,0)*+{R_0}="0";
(-32,0)*+{R_1}="1";
(-16,0)*+{R_2}="2";
(1,0)*+{\cdots}="3";
(20,0)*+{R_k}="4";
(39,0)*+{\cdots}="5";
(58,0)*+{R_r}="6";
(61,-1)*+{,}="7";
(-45,2)*+{}="04";
(-45,0)*+{}="05";
(-45,-2)*+{}="06";
(-35,2)*+{}="11";
(-35,0)*+{}="12";
(-35,-2)*+{}="13";
(-29,2)*+{}="14";
(-29,0)*+{}="15";
(-29,-2)*+{}="16";
(-19,2)*+{}="21";
(-19,0)*+{}="22";
(-19,-2)*+{}="23";
(-13,2)*+{}="24";
(-13,0)*+{}="25";
(-13,-2)*+{}="26";
(-3,2)*+{}="31";
(-3,0)*+{}="32";
(-3,-2)*+{}="33";
(5,2)*+{}="34";
(5,0)*+{}="35";
(5,-2)*+{}="36";
(17,2)*+{}="41";
(17,0)*+{}="42";
(17,-2)*+{}="43";
(23,2)*+{}="44";
(23,0)*+{}="45";
(23,-2)*+{}="46";
(35,2)*+{}="51";
(35,0)*+{}="52";
(35,-2)*+{}="53";
(43,2)*+{}="54";
(43,0)*+{}="55";
(43,-2)*+{}="56";
(55,2)*+{}="61";
(55,0)*+{}="62";
(55,-2)*+{}="63";
{\ar^{\ind} "04";"11"};
{\ar|*{_{\res}} "12";"05"};
{\ar_{\jnd} "06";"13"};
{\ar^{\ind} "14";"21"};
{\ar|*{_{\res}} "22";"15"};
{\ar_{\jnd} "16";"23"};
{\ar^{\ind} "24";"31"};
{\ar|*{_{\res}} "32";"25"};
{\ar_{\jnd} "26";"33"};
{\ar^{\ind} "34";"41"};
{\ar|(0.45)*{_{\res}} "42";"35"};
{\ar_{\jnd} "36";"43"};
{\ar^{\ind} "44";"51"};
{\ar|(0.45)*{_{\res}} "52";"45"};
{\ar_{\jnd} "46";"53"};
{\ar^{\ind} "54";"61"};
{\ar|(0.45)*{_{\res}} "62";"55"};
{\ar_{\jnd} "56";"63"};
\endxy
\]
where $R_k=\Omega_G(G/H_k)$.
With this identification, an ideal of $\Omega_G$ can be regarded as a sequence $[I_0,\ldots,I_r]$ of ideals $I_k\subseteq R_k$.
A sequence $[I_0,\ldots,I_r]$ forms an ideal of $\Omega_G$ if and only if the condition
\begin{itemize}
\item[\fbox{$\Ical(k)$}]
$\ \ \ind^k_{k-1}(I_{k-1})\subseteq I_k,\ \ 
\res^k_{k-1}(I_k)\subseteq I_{k-1},\ \ 
\jnd^k_{k-1}(I_{k-1})\subseteq I_k$.
\end{itemize}
is satisfied for each $1\le k\le r$ (Corollary \ref{CorIk}).
The restriction of an ideal $\I=[I_0,\ldots,I_k]$ of $\Omega_{H_k}$ onto $H_i$ is given by $\I|_{H_i}=[I_0,\ldots,I_i]$ for any $0\le i\le k$. On the contrary, we can consider an extension of $\I$ onto $H_{k+1}$, namely, an ideal $\I\ppr$ of $\Omega_{H_{k+1}}$ satisfying $\I\ppr |_{H_k}=\I$. In particular, the largest and the smallest among such $\I\ppr$ are explicitly given by $\Lcal\I$ and $\Scal\I$ in Definition \ref{DefLandS}.
This allows us an inductive construction of ideals in $\Omega_G$.

Whether an ideal $\I=[I_0,\ldots,I_r]$ is prime or not can be also checked inductively on $k$. In fact, it is prime if and only if the condition
\begin{itemize}
\item[\fbox{$\Pcal(k)$}] For any $0\le i\le k$, $a\in (\res^k_{k-1})^{-1}(I_{k-1})$ and $b\in (\res^i_{i-1})^{-1}(I_{i-1})\setminus I_i$, 
\[ a\cdot\jnd^k_i(b)\in I_k\ \ \Longrightarrow\ \ a\in I_k \]
holds.
\end{itemize}
is satisfied for each $0\le k\le r$ (Proposition \ref{PropEP}).
Consequently, any restriction $\I|_{H_i}$ of a prime ideal $\I=[I_0,\ldots,I_k]\subseteq\Omega_{H_k}$ is again prime. Especially $I_0$ should be a prime ideal in $R_0\cong\Z$, and thus equal to $0$, $(p)$, or $(q)$ for some prime integer $q\ne p$.
By determining prime ideals over $0, (p), (q)$, namely, prime ideals $\I=[I_0,\ldots,I_r]$ satisfying $I_0=0,(p),(q)$ respectively, we obtain the following result. In particular, the dimension of $\Omega_G$ is calculated as $\dim\Omega_G=r+1$ (Corollary \ref{CorDim}).

\begin{thm}
Let $G$ be a cyclic $p$-group of $p$-rank $r$. The prime ideals of $\Omega_G\in\Ob(\TamG)$ are as follows.
\begin{itemize}
\item[{\rm (i)}] Over $(q)\subseteq R_0$, there are $\ell+1$ prime ideals
\[ \Scal^r(q)\subsetneq\Lcal\Scal^{r-1}(q)\subsetneq\cdots\subsetneq\Lcal^r(q). \]
\item[{\rm (ii)}] Over $(p)\subseteq R_0$, there is only one prime ideal $\Lcal^r(p)$.
\item[{\rm (iii)}] Over $0\subseteq R_0$, there are $\ell+1$ prime ideals
\[ 0\subsetneq\Lcal(0)\subsetneq\cdots\subsetneq\Lcal^r(0). \]
\end{itemize}
Thus we have
\begin{eqnarray*}
\Spec\,\Omega_G&\!\!\! =&\!\!\!\!\{ \Lcal^r(p)\}\ \cup\ \{\Lcal^i(0)\mid 0\le i\le r\}\\
&&\!\!\cup\ \,\{ \Lcal^i\Scal^{r-i}(q)\mid 0\le i\le r,\ q\ \text{is a prime different from}\ p \}.
\end{eqnarray*}
\end{thm}

\bigskip

Throughout this article, the unit of a finite group $G$ will be denoted by $e$. Abbreviately we denote the trivial subgroup of $G$ by $e$, instead of $\{ e\}$.
$H\le G$ means $H$ is a subgroup of $G$.
$\Gs$ denotes the category of finite $G$-sets and $G$-equivariant maps.
If $H\le G$ and $g\in G$, then ${}^g\! H$ denotes the conjugate ${}^g\! H=gHg^{-1}$. Similarly, $H^{\! g}=g^{-1}Hg$.
A monoid is always assumed to be unitary and commutative. Similarly a ring is assumed to be commutative, with an additive unit $0$ and a multiplicative unit $1$.
We denote the category of monoids by $\Mon$, the category of rings by $\Ring$. 
A monoid homomorphism preserves units, and a ring homomorphism preserves $0$ and $1$.
For any category $\mathscr{C}$ and any pair of objects $X$ and $Y$ in $\mathscr{C}$, the set of morphisms from $X$ to $Y$ in $\mathscr{C}$ is denoted by $\mathscr{C}(X,Y)$.

\bigskip

\bigskip

We briefly recall the definition of a Tambara functor.

\begin{dfn}(\cite{Tam})\label{DefTamFtr}
A {\it Tambara functor} $T$ {\it on} $G$ is a triplet $T=(T^{\ast},T_+,T_{\bullet})$ of two covariant functors
\[ T_+\colon\Gs\rightarrow\Sett,\ \ T_{\bullet}\colon\Gs\rightarrow\Sett \]
and one contravariant functor
\[ T^{\ast}\colon\Gs\rightarrow\Sett \]
which satisfies the following. Here $\Sett$ is the category of sets.
\begin{enumerate}
\item $T^{\alpha}=(T^{\ast},T_+)$ is a Mackey functor on $G$.
\item $T^{\mu}=(T^{\ast},T_{\bullet})$ is a semi-Mackey functor on $G$.

\noindent Since $T^{\alpha},T^{\mu}$ are semi-Mackey functors, we have $T^{\ast}(X)=T_+(X)=T_{\bullet}(X)$ for each $X\in\Ob(\Gs)$. We denote this by $T(X)$.
\item (Distributive law)
If we are given an exponential diagram
\[
\xy
(-12,6)*+{X}="0";
(-12,-6)*+{Y}="2";
(0,6)*+{A}="4";
(12,6)*+{Z}="6";
(12,-6)*+{B}="8";
(0,0)*+{exp}="10";
{\ar_{f} "0";"2"};
{\ar_{p} "4";"0"};
{\ar_{\lambda} "6";"4"};
{\ar^{\rho} "6";"8"};
{\ar^{q} "8";"2"};
\endxy
\]
in $\Gs$, then
\[
\xy
(-18,7)*+{T(X)}="0";
(-18,-7)*+{T(Y)}="2";
(0,7)*+{T(A)}="4";
(18,7)*+{T(Z)}="6";
(18,-7)*+{T(B)}="8";
{\ar_{T_{\bullet}(f)} "0";"2"};
{\ar_{T_+(p)} "4";"0"};
{\ar^{T^{\ast}(\lambda)} "4";"6"};
{\ar^{T_{\bullet}(\rho)} "6";"8"};
{\ar^{T_+(q)} "8";"2"};
{\ar@{}|\circlearrowright "0";"8"};
\endxy
\]
is commutative. For the definition and basic properties of exponential diagrams, see \cite{Tam}.
\end{enumerate}

If $T=(T^{\ast},T_+,T_{\bullet})$ is a Tambara functor, then $T(X)$ becomes a ring for each $X\in\Ob(\Gs)$ (\cite{Tam}). Since $T^{\alpha}$ is a Mackey functor, by definition $T$ is \lq{\it additive}', in the sense that for any $X_1,X_2\in\Ob(\Gs)$, the inclusions $\iota_1\colon X_1\hookrightarrow X_1\amalg X_2$ and $\iota_2\colon X_2\hookrightarrow X_1\amalg X_2$ induce a natural isomorphism of rings $(T^{\ast}(\iota_1),T^{\ast}(\iota_2))\colon T(X_1\amalg X_2)\overset{\cong}{\longrightarrow}T(X_1)\times T(X_2)$.
For each $f\in\Gs(X,Y)$,
\begin{itemize}
\item $T^{\ast}(f)\colon T(Y)\rightarrow T(X)$ is a ring homomorphism, called the {\it restriction} along $f$, 
\item $T_+(f)\colon T(X)\rightarrow T(Y)$ is an additive homomorphism, called the {\it additive transfer} along $f$,
\item $T_{\bullet}(f)\colon T(X)\rightarrow T(Y)$ is a multiplicative homomorphism, called the {\it multiplicative transfer} along $f$.
\end{itemize}
$T^{\ast}(f),T_+(f),T_{\bullet}(f)$ are often abbreviated to $f^{\ast},f_+,f_{\bullet}$.
\end{dfn}

\begin{rem}\label{Remresindjnd}
If $f$ is the natural projection $\pro^H_K\colon G/K\rightarrow G/H$ for some $K\le H\le G$, then $f^{\ast},f_+,f_{\bullet}$ is written as
\begin{eqnarray*}
\res^H_K&=&(\pro^H_K)^{\ast}\ ,\\
\ind^H_K&=&(\pro^H_K)_+\,,\\
\jnd^H_K&=&(\pro^H_K)_{\bullet}\ .
\end{eqnarray*}
For a conjugate map $c_g\colon G/H^{\! g}\rightarrow G/H$, we define $c_{g,H}\colon T(G/H)\rightarrow T(G/H^{\! g})$ (or simply $c_g\colon T(G/H)\rightarrow T(G/H^{\! g})$) by
\[ c_{g,H}=T^{\ast}(c_g). \]
If $g$ belongs to the normalizer $N_G(H)$ of $H$ in $G$, then this gives an automorphism $c_{g,H}\colon T(G/H)\rightarrow T(G/H)$. With this $N_G(H)$-action, every $T(G/H)$ becomes an $N_G(H)/H$-ring.

Since any $G$-map is a union of compositions of natural projections and conjugate maps, the structure morphisms of a Tambara functor are completely determined by $\res^H_K,\ind^H_K,\jnd^H_K,c_{g,H}$ for $K\le H\le G, g\in G$, by virtue of the additivity.
\end{rem}


\begin{ex}\label{ExTam}
If we define $\Omega_G$ by
\[ \Omega_G(X)=K_0(\Gs/X) \]
for each $X\in\Ob(\Gs)$, where the right hand side is the Grothendieck ring of the category of finite $G$-sets over $X$, then $\Omega_G$ becomes a Tambara functor on $G$ \cite{Tam}. This is called the {\it Burnside Tambara functor}.
For each $f\in\Gs(X,Y)$,
\[ f_{\bullet}\colon\Omega_G(X)\rightarrow\Omega_G(Y) \]
is the one determined by
\[ \qquad\quad f_{\bullet}(A\overset{p}{\rightarrow}X)=(\Pi_f(A)\overset{\varpi}{\rightarrow}Y)\quad\  ({}^{\forall}(A\overset{p}{\rightarrow}X)\in\Ob(\Gs/X)), \]
where $\Pi_f(A)$ and $\varpi$ is
\[ \Pi_f(A)=\Set{ (y,\sigma)| \begin{array}{l}y\in Y,\\ \sigma\colon f^{-1}(y)\rightarrow A\ \ \text{is a map of sets},\\ p\circ\sigma=\id_{f^{-1}(y)}\end{array} }, \]

\[ \varpi(y,\sigma)=y.\hspace{7.3cm} \]
$G$ acts on $\Pi_f(A)$ by $g\cdot(y,\sigma)=(gy,{}^g\!\sigma)$, where ${}^g\!\sigma$ is the map defined by
\[ {}^g\!\sigma(x)=g\sigma(g^{-1}x)\quad({}^{\forall}x\in f^{-1}(gy)). \]
$f_+\colon\Omega_G(X)\rightarrow\Omega_G(Y)$ is an additive homomorphism satisfying
\[ f_+(A\overset{p}{\rightarrow}X)=(A\overset{f\circ p}{\rightarrow}Y)\quad\  ({}^{\forall}(A\overset{p}{\rightarrow}X)\in\Ob(\Gs/X)). \]
$f^{\ast}$ is defined by using a fiber product (\cite{Tam}). Namely, the Mackey-functor structure on the additive part of $\Omega$ is the usual one as in \cite{Bouc}.
\begin{rem}\label{RemB}
For any $K\le H\le G$, we have a natural isomorphism (cf.$\,\,$\cite{Bouc}) $\Omega_G(G/K)\cong\Omega_H(H/K)$, and we will identify them through this isomorphism. 
\end{rem}
We often abbreviate $\Omega_G$ to $\Omega$, if the base group is obvious from the context.
\end{ex}

\section{Burnside Tambara functor on a cyclic $p$-group}

Throughout this article, we fix a prime number $p$.
Let $G$ be a cyclic $p$-group of $p$-rank $r\ge0$,
and let $H_k\le G$ be its subgroup of order $p^k$
for each $0\le k\le r$.
In the following argument, without loss of generality we may assume
\begin{eqnarray*}
&G=\Z/p^r\Z,\qquad
H_k=p^{r-k}\Z/p^r\Z,&\\
&e=H_0<H_1\cdots <H_k<\cdots <H_r=G.&
\end{eqnarray*}
Then the $G$-set $G/H_k$ is canonically isomorphic to $\Z/p^{r-k}\Z$, and the natural projection
\[ G/H_k\overset{\pro^{\ell}_k}{\longrightarrow}G/H_{\ell}\quad (k\le\ell) \]
is identified with a map given by
\[ \Z/p^{r-k}\Z\rightarrow\Z/p^{r-\ell}\Z\ ;\ \ a\ \mathrm{mod}\ p^{r-k}\Z\ \mapsto\ a\ \mathrm{mod}\ p^{r-\ell}\Z \]
for any $a\in\Z$.

Since $G$ is commutative, each $\Omega(G/H_k)$ has a trivial $G$-action, and admits a natural $\Z$-basis
\[ \{ (G/H_i\overset{\pro^k_i}{\longrightarrow}G/H_k)\mid 0\le i\le k \}. \]
Thus if we denote $(G/H_i\overset{\pro^k_i}{\longrightarrow}G/H_k)$ by $X_{k,i}$, then it is a free $\Z$-module
\[ \Omega(G/H_k)=\bigoplus_{0\le i\le k}\Z X_{k,i} \]
with a trivial $G$-action.
Therefore, if we put $R_k=\displaystyle\bigoplus_{0\le i\le k} \Z X_{k,i}$, then the Tambara functor $\Omega_G$ is regarded just as a sequence of commutative rings $R_k$ and structure morphisms
\[
\xy
(-48,0)*+{R_0}="0";
(-32,0)*+{R_1}="1";
(-16,0)*+{R_2}="2";
(1,0)*+{\cdots}="3";
(20,0)*+{R_k}="4";
(39,0)*+{\cdots}="5";
(58,0)*+{R_r}="6";
(-45,2)*+{}="04";
(-45,0)*+{}="05";
(-45,-2)*+{}="06";
(-35,2)*+{}="11";
(-35,0)*+{}="12";
(-35,-2)*+{}="13";
(-29,2)*+{}="14";
(-29,0)*+{}="15";
(-29,-2)*+{}="16";
(-19,2)*+{}="21";
(-19,0)*+{}="22";
(-19,-2)*+{}="23";
(-13,2)*+{}="24";
(-13,0)*+{}="25";
(-13,-2)*+{}="26";
(-3,2)*+{}="31";
(-3,0)*+{}="32";
(-3,-2)*+{}="33";
(5,2)*+{}="34";
(5,0)*+{}="35";
(5,-2)*+{}="36";
(17,2)*+{}="41";
(17,0)*+{}="42";
(17,-2)*+{}="43";
(23,2)*+{}="44";
(23,0)*+{}="45";
(23,-2)*+{}="46";
(35,2)*+{}="51";
(35,0)*+{}="52";
(35,-2)*+{}="53";
(43,2)*+{}="54";
(43,0)*+{}="55";
(43,-2)*+{}="56";
(55,2)*+{}="61";
(55,0)*+{}="62";
(55,-2)*+{}="63";
{\ar^{\ind^1_0} "04";"11"};
{\ar|*{_{\res^1_0}} "12";"05"};
{\ar_{\jnd^1_0} "06";"13"};
{\ar^{\ind^2_1} "14";"21"};
{\ar|*{_{\res^2_1}} "22";"15"};
{\ar_{\jnd^2_1} "16";"23"};
{\ar^{\ind^3_2} "24";"31"};
{\ar|*{_{\res^3_2}} "32";"25"};
{\ar_{\jnd^3_2} "26";"33"};
{\ar^{\ind^k_{k-1}} "34";"41"};
{\ar|(0.45)*{_{\res^k_{k-1}}} "42";"35"};
{\ar_{\jnd^k_{k-1}} "36";"43"};
{\ar^{\ind^{k+1}_k} "44";"51"};
{\ar|(0.45)*{_{\res^{k+1}_k}} "52";"45"};
{\ar_{\jnd^{k+1}_k} "46";"53"};
{\ar^{\ind^r_{r-1}} "54";"61"};
{\ar|(0.45)*{_{\res^r_{r-1}}} "62";"55"};
{\ar_{\jnd^r_{r-1}} "56";"63"};
\endxy
\]
satisfying conditions in Definition \ref{DefTamFtr}.
Here, $\ind^k_{k-1},\res^k_{k-1},\jnd^k_{k-1}$ are the abbreviations of $\ind^{H_k}_{H_{k-1}},\res^{H_k}_{H_{k-1}},\jnd^{H_k}_{H_{k-1}}$. We use similar abbreviations in the rest. Remark that any structure morphism of $\Omega_G$ can be realized as a composition of these morphisms.

\begin{rem}\label{RemSeq}
By virtue of Remark \ref{RemB}, the first $k$-terms
\begin{equation}\label{k-seq}
\xy
(-48,0)*+{R_0}="0";
(-32,0)*+{R_1}="1";
(-16,0)*+{R_2}="2";
(1,0)*+{\cdots}="3";
(20,0)*+{R_k}="4";
(-45,2)*+{}="04";
(-45,0)*+{}="05";
(-45,-2)*+{}="06";
(-35,2)*+{}="11";
(-35,0)*+{}="12";
(-35,-2)*+{}="13";
(-29,2)*+{}="14";
(-29,0)*+{}="15";
(-29,-2)*+{}="16";
(-19,2)*+{}="21";
(-19,0)*+{}="22";
(-19,-2)*+{}="23";
(-13,2)*+{}="24";
(-13,0)*+{}="25";
(-13,-2)*+{}="26";
(-3,2)*+{}="31";
(-3,0)*+{}="32";
(-3,-2)*+{}="33";
(5,2)*+{}="34";
(5,0)*+{}="35";
(5,-2)*+{}="36";
(17,2)*+{}="41";
(17,0)*+{}="42";
(17,-2)*+{}="43";
{\ar^{\ind^1_0} "04";"11"};
{\ar|*{_{\res^1_0}} "12";"05"};
{\ar_{\jnd^1_0} "06";"13"};
{\ar^{\ind^2_1} "14";"21"};
{\ar|*{_{\res^2_1}} "22";"15"};
{\ar_{\jnd^2_1} "16";"23"};
{\ar^{\ind^3_2} "24";"31"};
{\ar|*{_{\res^3_2}} "32";"25"};
{\ar_{\jnd^3_2} "26";"33"};
{\ar^{\ind^k_{k-1}} "34";"41"};
{\ar|(0.45)*{_{\res^k_{k-1}}} "42";"35"};
{\ar_{\jnd^k_{k-1}} "36";"43"};
\endxy
\end{equation}
can be regarded as a sequence representing the Tambara functor $\Omega_{H_k}$. We always work under this identification in this paper.
With this identification, forgetting the entire group $G$, we can regard the Burnside Tambara functor $\Omega_{H_k}$ on a cyclic $p$-group $H_k$ of $p$-rank $k$, simply as a length $k$ sequence of rings $(\ref{k-seq})$ obtained inductively by adding $R_k,\ind^k_{k-1},\res^k_{k-1},\jnd^k_{k-1}$
to the length $k-1$ sequence
\[
\xy
(-32,0)*+{R_0}="1";
(-16,0)*+{R_1}="2";
(1,0)*+{\cdots}="3";
(23,0)*+{R_{k-1}}="4";
%
%
%
(-29,2)*+{}="14";
(-29,0)*+{}="15";
(-29,-2)*+{}="16";
(-19,2)*+{}="21";
(-19,0)*+{}="22";
(-19,-2)*+{}="23";
(-13,2)*+{}="24";
(-13,0)*+{}="25";
(-13,-2)*+{}="26";
(-3,2)*+{}="31";
(-3,0)*+{}="32";
(-3,-2)*+{}="33";
(5,2)*+{}="34";
(5,0)*+{}="35";
(5,-2)*+{}="36";
(18,2)*+{}="41";
(18,0)*+{}="42";
(18,-2)*+{}="43";
%
%
{\ar^{\ind^1_0} "14";"21"};
{\ar|*{_{\res^1_0}} "22";"15"};
{\ar_{\jnd^1_0} "16";"23"};
{\ar^{\ind^2_1} "24";"31"};
{\ar|*{_{\res^2_1}} "32";"25"};
{\ar_{\jnd^2_1} "26";"33"};
{\ar^{\ind^{k-1}_{k-2}} "34";"41"};
{\ar|(0.45)*{_{\res^{k-1}_{k-2}}} "42";"35"};
{\ar_{\jnd^{k-1}_{k-2}} "36";"43"};
\endxy
\qquad\quad
\]
corresponding to $\Omega_{H_{k-1}}$.
This observation enables us an inductive construction of ideals of in $\Omega$.
\end{rem}

\smallskip

We go on to describe $R_k$ and the structure morphisms. As above, we have $R_k=\underset{0\le i\le k}{\bigoplus}\Z X_{k,i}$ as a module.
\begin{rem}\label{Remind}
Additive transfer $\ind^{\ell}_k\colon R_k\rightarrow R_{\ell}$ is given by
\[ \ind^{\ell}_k(X_{k,i})=X_{\ell,i}\quad(0\le i\le k) \]
for each $k\le \ell$.
\end{rem}
\begin{proof}
This is obvious.
\end{proof}

If we take a fiber product of $\pro^k_i\colon G/H_i\rightarrow G/H_k$ and $\pro^k_j\colon G/H_j\rightarrow G/H_k$
\[
\xy
(-12,5.6)*+{G/H_i\underset{G/H_k}{\times}G/H_j}="0";
(1,7)*+{}="1";
(12,7)*+{G/H_j}="2";
(-12,3)*+{}="3";
(-12,-7)*+{G/H_i}="4";
(12,-7)*+{G/H_k}="6";
(0,0)*+{\square}="8";
{\ar^{} "1";"2"};
{\ar_{} "3";"4"};
{\ar^{\pro^k_j} "2";"6"};
{\ar_{\pro^k_i} "4";"6"};
\endxy
\]
where $0\le i,j\le k$, then by the Mackey decomposition formula, we have
\begin{equation}\label{Fibij}
G/H_i\underset{G/H_k}{\times}G/H_j=%
\begin{cases}
\ \underset{p^{k-j}}{\coprod}G/H_i&(i\le j)\\
\ \underset{p^{k-i}}{\coprod}G/H_j&(j\le i).
\end{cases}
\end{equation}
In any case, the cardinality of this $G$-set is
\[ |G/H_i\underset{G/H_k}{\times}G/H_j|=p^{r+k-(i+j)}. \]
Equation $(\ref{Fibij})$ is also written as
\[ G/H_i\underset{G/H_k}{\times}G/H_j=\underset{p^{k-\max(i,j)}}{\coprod}G/H_{\min(i,j)}. \]
As a corollary, we obtain the following.
\begin{cor}\label{CorRk}
For any $0\le k\le r$, the following holds.
\begin{enumerate}
\item For any $i,j\le k$,
\[ X_{k,i}\cdot X_{k,j}=%
\begin{cases}
\ p^{k-j}X_{k,i}&(i\le j)\\
\ p^{k-i}X_{k,j}&(j\le i)
\end{cases}
\]
in $R_k$. In fact, as a residue ring of the polynomial ring $\Z[X_{k,i}\mid0\le  i\le k]$ over indeterminates $X_{k,i}\ (0\le i\le k)$, $R_k$ is written as
\[ \Z[X_{k,i}\mid 0\le i\le k]/(\{ X_{k,i}X_{k,j}-p^{k-\max(i,j)}X_{k,\min(i,j)}\mid 0\le i,j\le k \}). \]
In particular, $X_{k,k}=1$ is the unit of $R_k$.
\item For any $\ell\ge k$, $\res^{\ell}_k\colon R_{\ell}\rightarrow R_k$ is given by
\[ \res^{\ell}_k(X_{\ell,i})=%
\begin{cases}
\ p^{\ell-k}X_{k,i}&(i\le k)\\
\ p^{\ell-i}X_{k,k}=p^{\ell-i}&(i\ge k).
\end{cases}
\]
\end{enumerate}
\end{cor}
As for multiplicative transfers, we need some calculation. For the detail, see the Appendix.
\begin{cor}\label{Corjnd}
For any $\ell\ge k$, $\jnd^{\ell}_k\colon R_k\rightarrow R_{\ell}$ is given by
\begin{eqnarray*}
\jnd^{\ell}_k(\sum_{0\le i\le k}m_iX_{k,i})&=&m_kX_{\ell,\ell}%
+\sum_{k\le i< \ell}\frac{(m_k)^{p^{\ell-i}}-(m_k)^{p^{\ell-i-1}}}{p^{\ell-i}}X_{\ell,i}\\
&+&\sum_{0\le i<k}\frac{(\sum^k_{s=i}m_sp^{k-s})^{p^{\ell-k}}-(\sum^k_{s=i+1}m_sp^{k-s})^{p^{\ell-k}}}{p^{\ell-i}}X_{\ell,i}
\end{eqnarray*}
for any $m_0,\ldots,m_k\in\Z$.
\end{cor}

\section{Ideals of $\Omega$ as a sequence}
An ideal of a Tambara functor is defined in \cite{N_Ideal} as follows.
\begin{dfn}\label{DefIdeal}
Let $T$ be a Tambara functor. An {\it ideal} $\I$ of $T$ is a family of ideals $\{\I(X)\subseteq T(X)\}_{X\in\Ob(\Gs)}$ satisfying
\begin{enumerate}
\item[{\rm (i)}] $f^{\ast}(\I(Y))\subseteq \I(X)$,
\item[{\rm (ii)}] $f_+(\I(X))\subseteq \I(Y)$,
\item[{\rm (iii)}] $f_{\bullet}(\I(X))\subseteq f_{\bullet}(0)+\I(Y)$
\end{enumerate}
for any $f\in\Gs(X,Y)$.
\end{dfn}
These conditions also imply
\[ \I(X_1\amalg X_2)\cong\I(X_1)\times\I(X_2) \]
for any $X_1,X_2\in\Ob(\Gs)$. From this, an ideal of $T$ is determined by a family
\[ \{ I_H=\I(G/H) \}_{H\le G} \]
of ideals $I_H\subseteq T(G/H)$, indexed by the set of subgroups of $G$, as follows.
\begin{prop}\label{PropIC}
Let $T$ be a Tambara functor on $G$. To give an ideal $\I$ of $T$ is equivalent to give a family $\{ I_H\}_{H\le G}$ of ideals $I_H\subseteq T(G/H)$ satisfying the conditions
\begin{itemize}
\item[{\rm (i)}] $\res^H_K(I_H)\subseteq I_K$
\item[{\rm (ii)}] $\ind^H_K(I_K)\subseteq I_H$
\item[{\rm (iii)}] $\jnd^H_K(I_K)\subseteq I_H$
\item[{\rm (iv)}] $c_{g,H}(I_H)\subseteq I_{H^{\! g}}$
\end{itemize}
for any $K\le H\le G$ and $g\in G$.
\end{prop}
\begin{proof}
This is straightforward (cf. Corollary 2.2 in \cite{N_Specp}).
\end{proof}

In our particular case, an ideal of $\Omega_G$ is written as follows.

\begin{cor}\label{CorIk}
Let $G$ be a cyclic $p$-group of $p$-rank $r$. An ideal $\I$ of $\Omega_G$ is given by a sequence
\[ \I=[I_0,\ldots,I_r] \]
of ideals $I_k\subseteq R_k$, satisfying the following condition $\Ical(k)$ for each $1\le k\le r$.
\medskip

\begin{itemize}
\item[\fbox{$\Ical(k)$}]
$\ \ \ind^k_{k-1}(I_{k-1})\subseteq I_k,\ \ 
\res^k_{k-1}(I_k)\subseteq I_{k-1},\ \ 
\jnd^k_{k-1}(I_{k-1})\subseteq I_k$.
\end{itemize}
\end{cor}

\bigskip

Remark that for ideals $\I=[I_0,\ldots,I_r]$ and $\J=[J_0,\ldots,J_r]$, we have $\I\subseteq\J$ if and only if $I_k\subseteq J_k$ holds for any $0\le k\le r$.

In the following, we will simply say
\lq\lq $[I_0,\ldots,I_k]$ is an ideal of $\Omega_{H_k}$"
to mean that $[I_0,\ldots I_k]$ satisfies $\Ical(i)$ for any $1\le i\le k$.
This makes sense by virtue of Remark \ref{RemSeq}.
The restriction of an ideal $\I=[I_0\ldots,I_k]$ of $\Omega_{H_k}$ onto $H_i$ is given by $[I_0\ldots,I_i]$, for each $0\le i\le k$. We denote this by $\I|_{H_i}$. Obviously, restriction preserves inclusions of ideals.

On the contrary, Corollary \ref{CorIk} also enables us to extend an ideal $[I_0,\ldots,I_{k-1}]$ of $\Omega_{H_{k-1}}$ to an ideal $[I_0\ldots,I_k]$ of $\Omega_{H_k}$, by adding an ideal $I_k\subseteq R_k$ satisfying $\Ical(k)$.
Among the possible extensions of $[I_0\ldots,I_{k-1}]$, the largest and the smallest are given as follows. If an ideal $\I\ppr\subseteq\Omega_{H_k}$ satisfies $\I\ppr |_{H_i}=\I$ for a given ideal $\I\subseteq\Omega_{H_i}$, then we say that $\I\ppr$ {\it is an ideal over} $\I$.
\begin{prop}\label{PropLandS}
For a fixed $1\le k\le r$, suppose we are given an ideal $\I=[I_0\ldots,I_{k-1}]$ of $\Omega_{H_{k-1}}$.
\begin{enumerate}
\item If we define $L_k(I_{k-1})=L(I_{k-1})\subseteq R_k$ by
\[ L(I_{k-1})=(\res^k_{k-1})^{-1}(I_{k-1}), \]
then $[I_0\ldots,I_{k-1},L(I_{k-1})]$ is the largest ideal of $\Omega_{H_k}$ over $[I_0\ldots,I_{k-1}]$. 
\item If we define $S_k(I_{k-1})=S(I_{k-1})\subseteq R_k$ to be the ideal of $R_k$ generated by
\[ \{\ind^k_{k-1}(\alpha)\mid\alpha\in I_{k-1}\}\ \ \text{and}\ \ \{\jnd^k_{k-1}(\alpha)\mid\alpha\in I_{k-1}\}, \]
shortly,
\[ S(I_{k-1})=(\ind^k_{k-1}(I_{k-1}))+(\jnd^k_{k-1}(I_{k-1})), \]
then $[I_0\ldots,I_{k-1},S(I_{k-1})]$ is the smallest ideal of $\Omega_{H_k}$ over $[I_0\ldots,I_{k-1}]$. 
\end{enumerate}
\end{prop}

\begin{dfn}\label{DefLandS}
For an ideal $\I=[I_0,\ldots,I_{k-1}]$ of $\Omega_{H_{k-1}}$, define ideals $\Lcal_k(\I)=\Lcal\I$ and $\Scal_k(\I)=\Scal\I$ of $\Omega_{H_k}$ by
\begin{eqnarray*}
\Lcal\I&=&[I_0,\ldots,I_{k-1},L(I_{k-1})],\\
\Scal\I&=&[I_0,\ldots,I_{k-1},S(I_{k-1})],
\end{eqnarray*}
where $L(I_{k-1})$ and $S(I_{k-1})$ are those in Proposition \ref{PropLandS}.
Additionally, we denote the $n$-times iterations of $L,S,\Lcal,\Scal$ by $L^n,S^n$ and $\Lcal^n,\Scal^n$. For example, an ideal $\I=[I_0,\ldots,I_k]\subseteq\Omega_{H_k}$ yields $\Lcal^n\I=[I_0,\ldots,I_k,L(I_k),\ldots,L^n(I_k)]\ \subseteq\Omega_{H_{k+n}}$.
\end{dfn}

\begin{proof}[Proof of Proposition \ref{PropLandS}]
\noindent{\rm (1)}
To show $\Lcal\I$ is an ideal of $\Omega_{H_k}$, it suffices to confirm $\Ical(k)$ is satisfied.
By definition, $\res^k_{k-1}(L(I_{k-1}))\subseteq I_{k-1}$ is obvious.
In addition, by the existence of a pullback diagram
\[
\xy
(-13,6)*+{\underset{p}{\amalg}\ \! G/H_{k-1}}="0";
(-5,7)*+{}="1";
(12,7)*+{G/H_{k-1}}="2";
(-12,5)*+{}="3";
(-12,-7)*+{G/H_{k-1}}="4";
(12,-7)*+{G/H_k}="6";
(0,0)*+{\square}="8";
{\ar^(0.3){\nabla} "1";"2"};
{\ar_{} "3";"4"};
{\ar^{\pro^k_{k-1}} "2";"6"};
{\ar_(0.55){\pro^k_{k-1}} "4";"6"};
\endxy
\]
where $\nabla\colon\underset{p}{\amalg}G/H_{k-1}\rightarrow G/H_{k-1}$ is the folding map, we have
\begin{eqnarray*}
\res^k_{k-1}\circ\ind^k_{k-1}(\alpha)&=&p\alpha,\\
\res^k_{k-1}\circ\jnd^k_{k-1}(\alpha)&=&\alpha^p
\end{eqnarray*}
for any $\alpha\in R_k$.
Thus if $\alpha\in I_{k-1}$, then
\begin{eqnarray*}
&\ind^k_{k-1}(\alpha)\in(\res^k_{k-1})^{-1}(I_{k-1}),&\\
&\jnd^k_{k-1}(\alpha)\in(\res^k_{k-1})^{-1}(I_{k-1})\, &
\end{eqnarray*}
holds. This shows $\ind^k_{k-1}(I_{k-1})\subseteq L(I_{k-1})$ and $\jnd^k_{k-1}(I_{k-1})\subseteq L(I_{k-1})$, and thus $\Lcal\I$ is an ideal of $\Omega_{H_k}$. Moreover, since any ideal $[I_0,\ldots,I_{k-1},J_k]$ of $\Omega_{H_k}$ should satisfy $\res^k_{k-1}(J_k)\subseteq I_{k-1}$, obviously $\Lcal\I$ is the largest.

\bigskip

\noindent{\rm (2)}
It suffices to confirm $\Scal\I$ satisfies $\Ical(k)$.
Since $\res^k_{k-1}$ is a ring homomorphism, $\res^k_{k-1}(S(I_{k-1}))\subseteq I_{k-1}$ follows from the fact that any $\alpha\in I_{k-1}$ satisfies
\begin{eqnarray*}
\res^k_{k-1}\circ\ind^k_{k-1}(\alpha)=p\alpha&\in& I_{k-1},\\
\res^k_{k-1}\circ\jnd^k_{k-1}(\alpha)=\alpha^p&\in& I_{k-1}.
\end{eqnarray*}
The other conditions 
\begin{eqnarray*}
&\ind^k_{k-1}(I_{k-1})\subseteq S(I_{k-1}),&\\
&\jnd^k_{k-1}(I_{k-1})\subseteq S(I_{k-1})\, &
\end{eqnarray*}
are obviously satisfied by the definition of $S(I_{k-1})$.
Thus $\Scal\I$ is an ideal of $\Omega_{H_k}$.
Moreover, since any ideal $[I_0,\ldots,I_{k-1},J_k]$ of $\Omega_{H_k}$ should satisfy $\ind^k_{k-1}(I_{k-1})\subseteq J_k$ and $\jnd^k_{k-1}(I_{k-1})\subseteq J_k$, obviously $\Scal\I$ is the smallest.
\end{proof}

\section{Inductive criterion of primeness}

In \cite{N_Ideal}, a prime ideal of a Tambara functor is defined as follows.

\begin{dfn}\label{DefPrime}
Let $G$ be an arbitrary finite group, and let $T$ be a Tambara functor on $G$. An ideal $\I\subsetneq T$ is prime if and only if the following two conditions become equivalent, for any transitive $X,Y\in\Ob(\Gs)$ and any $a\in T(X)$, $b\in T(Y)$.
\begin{enumerate}
\item For any $C\in\Ob(\Gs)$ and for any pair of diagrams
\[ C\overset{v}{\leftarrow}D\overset{w}{\rightarrow}X,\quad C\overset{v\ppr}{\leftarrow}D\ppr\overset{w\ppr}{\rightarrow}Y \]
in $\Gs$, $\I(C)$ satisfies
\[ (v_{\bullet}w^{\ast}(a)-v_{\bullet}(0))\cdot(v\ppr_{\bullet}w^{\prime\ast}(b)-v\ppr_{\bullet}(0))\in\I(C). \]
\item $a\in\I(X)$ or $b\in\I(Y)$.
\end{enumerate}
Remark that {\rm (2)} always implies {\rm (1)}.
\end{dfn}

By a straightforward argument, we may assume $C,D,D\ppr$ are transitive, and this condition can be also written as follows.
\begin{prop}\label{PropPC}
Let $G$ be an arbitrary finite group, and let $\I=\{ I_H\}_{H\le G}$ be an ideal of $T$.
Then $\I$ is prime if and only if the following condition is satisfied for any $H,H\ppr\le G$ and any $a\in T(G/H),b\in T(G/H\ppr)$.
\begin{itemize}
\item[$(\ast)$]\label{PrimeCond}
If $\ (\jnd^L_{K^{\! g}}\circ c_{g,K}\circ\res^H_K(a))\cdot(\jnd^L_{K^{\prime g\ppr}}\circ c_{g\ppr,K\ppr}\circ\res^{H\ppr}_{K\ppr}(a))\,\in I_L$
is satisfied for any $L,K,K\ppr\le G$ and $g,g\ppr\in G$ satisfying $L\ge K^{\! g},\, K\le H,\, L\ge K^{\prime g\ppr},\, K\ppr\le H\ppr$, then
\[ a\in I_H \quad\text{or}\quad b\in I_{H\ppr} \]
holds.
\end{itemize}
\end{prop}

In our case, this can be reduced to the following condition.
\begin{cor}\label{CorPC}
Let $G$ be a cyclic $p$-group of $p$-rank $r\ge 0$.
An ideal $\I=[I_0,\ldots,I_r]$ of $\Omega_G$ is prime if and only if it satisfies the following condition $\Pcal(k,\ell)$ for each $0\le \ell\le k\le r$.
\begin{itemize}
\item[\fbox{$\Pcal(k,\ell)$}] For any $a\in R_k$ and $b\in R_{\ell}$,
\begin{eqnarray*}
&(\jnd^m_i\circ\res^k_i(a))\cdot(\jnd^m_j\circ\res^{\ell}_j(b))\in I_m&\\
&(0\le {}^{\forall}i\le k,\ 0\le {}^{\forall}j\le\ell,\ m=\max(i,j))&
\end{eqnarray*}
implies
\[ a\in I_k \quad\text{or}\quad b\in I_{\ell}. \]
\end{itemize}
\end{cor}
\begin{proof}
For any $0\le k,\ell\le r$ and any $a\in R_k,b\in R_{\ell}$, the condition $(\ast)$ in Proposition \ref{PropPC} is equivalent to the following.
\begin{itemize}
\item[$(\ast)$] If
\begin{equation}
\label{Eq_TT}
(\jnd^m_i\circ c_g\circ\res^k_i(a))\cdot(\jnd^m_j\circ c_{g\ppr}\circ\res^{\ell}_j(b))\in I_m
\end{equation}
is satisfied for any $g,g\ppr\in G$ and any $m\ge i\le k$,\ $m\ge j\le \ell$, then
\[ a\in I_k\quad\text{or}\quad b\in I_{\ell} \]
holds.
\end{itemize}
Remark that we have $c_g=\id,c_{g\ppr}=\id$. Besides, by {\rm (iii)} of Proposition \ref{PropIC}, assumption $(\ref{Eq_TT})$ is only have to be confirmed for $m=\max(i,j)$. Moreover, by the symmetry in $k$ and $\ell$, we may assume $\ell\le k$.
\end{proof}

Furthermore, this condition can be checked on each $k$-th step, as follows.

\begin{prop}\label{PropEP}
Let $G$ be as above. An ideal $\I=[I_0,\ldots,I_r]$ of $\Omega_G$ is prime if and only if it satisfies the following condition $\Pcal(k)$ for each $0\le k\le r$.
\begin{itemize}
\item[\fbox{$\Pcal(k)$}] For any $0\le i\le k$ and any $a\in L_k(I_{k-1}),b\in L_i(I_{i-1})\setminus I_i$,
\begin{equation}\label{ajnI}
a\cdot\jnd^k_i(b)\in I_k\ \ \Longrightarrow\ \ a\in I_k
\end{equation}
holds.
\end{itemize}
Here, when $k=0$, we define $L_0(I_{-1})$ to be $R_0$. Namely, $\Pcal(0)$ is as follows.
\begin{itemize}
\item[\fbox{$\Pcal(0)$}] For any $a\in R_0$ and $b\in R_0\setminus I_0$,
\[ ab\in I_0\ \ \Longrightarrow\ \ a\in I_0 \]
holds. $($This is saying $I_0\subseteq R_0$ is prime in the ordinary ring-theoretic meaning.$)$

\end{itemize}
\end{prop}
\begin{proof}
For each $0\le k\le r$, we define condition $\Qcal(k)$ as follows.

\smallskip

\begin{itemize}
\item[\fbox{$\Qcal(k)$}] $\Pcal(k,\ell)$ holds for all $0\le \ell\le k$.
\end{itemize}

\smallskip

It suffices to show that $\Qcal(k)$ holds for any $0\le k\le r$ if and only if $\Pcal(k)$ holds for any $0\le k\le r$.
This follows from:
\begin{claim}\label{ClaimPrimeEquiv}
For any $0\le k\le r$, the following holds.
\begin{enumerate}
\item $\Qcal(k)$ implies $\Pcal(k)$.
\item If $k\ge 1$, then $\Qcal(k-1)$ and $\Pcal(k)$ imply $\Qcal(k)$.
\item $\Qcal(0)$ is equivalent to $\Pcal(0)$.
\end{enumerate}
\end{claim}
In fact if this is shown, then by an induction on $k$, we can easily show that the following are equivalent for each $0\le k\le r$.
\begin{itemize}
\item[-] $\Qcal(k\ppr)$ holds for any $0\le k\ppr\le k$.
\item[-] $\Pcal(k\ppr)$ holds for any $0\le k\ppr\le k$.
\end{itemize}
This proves Proposition \ref{PropEP}.
Thus it remains to show Claim \ref{ClaimPrimeEquiv}.
\begin{proof}[Proof of Claim \ref{ClaimPrimeEquiv}]

\noindent{\rm (3)} When $k=0$, then the condition
\begin{itemize}
\item[\fbox{$\Qcal(0)$}] $\Pcal(0,0)$ holds. Namely, for any $a,b\in R_0$,
\[ ab\in I_0\ \ \Longrightarrow\ \ a\in I_0\ \ \text{or}\ \ b\in I_0 \]
holds.
\end{itemize}
is obviously equivalent to $\Pcal(0)$. 

\smallskip

\noindent{\rm (1)}
Fix $k$, suppose we are given $0\le\ell\le k$ and $a\in L_k(I_{k-1}),b\in L_{\ell}(I_{\ell-1})\setminus I_{\ell}$ satisfying
\begin{equation}\label{assump_remain}
a\cdot\jnd^k_{\ell}(b)\in I_k.
\end{equation}
It suffices to show $a\in I_k$.
Since $\Qcal(k)$ (in particular $\Pcal(k,\ell)$) is assumed, it is enough to confirm that
\begin{equation}\label{condstar}
(\jnd^m_i\circ\res^k_i(a))\cdot(\jnd^m_j\circ\res^{\ell}_j(b))\in I_m
\end{equation}
is satisfied for any $0\le i\le k,0\le j\le \ell$ and $m=\max(i,j)$.
However, when $i<k$ or $j<\ell$, $(\ref{condstar})$ follows from $\res^k_i(a)\in I_i$ and $\res^{\ell}_j(b)\in I_j$, since we have $a\in L_k(I_{k-1})$ and $b\in L_{\ell}(I_{\ell-1})$. 
In the remaining case of $i=k$ and $j=\ell$, 
$(\ref{condstar})$ is also satisfied by the assumption $(\ref{assump_remain})$.

\noindent{\rm (2)} Fix $k\ge 1$, and assume $\Qcal(k-1)$. Under this assumption, we show 
$\Pcal(k)$ implies $\Pcal(k,\ell)$ for any $0\le \ell \le k$.
By an induction on $\ell$, this is reduced to the following.
\begin{claim}\label{ClaimPl}
For any $1\le k\le r$ and $0\le \ell\le k$, we have
\[ \Qcal(k-1),\Pcal(k),\Pcal(k,\ell-1)\ \ \Longrightarrow\ \ \Pcal(k,\ell). \](Here, for $\ell=0$, $\Pcal(k,-1)$ is regarded as an empty condition.)
\end{claim}
We only have to show this claim in the rest.
Suppose $a\in R_k$ and $b\in R_{\ell}$ satisfy
\begin{equation}\label{EqTemp}
(\jnd^m_i\circ\res^k_i(a))\cdot(\jnd^m_j\circ\res^{\ell}_j(b))\in I_m
\end{equation}
for any $0\le i\le k$, $0\le j\le \ell$, $m=\max(i,j)$.
Claim \ref{ClaimPl} will follow immediately, if the following are shown.
\begin{itemize}
\item[{\rm (A)}]
If $a\in R_k\setminus I_k$, then $b\in L_{\ell}(I_{\ell-1})$. (
This is trivial when $\ell=0$, since we have defined as $L_0(I_{-1})=R_0$.)
\item[{\rm (B)}]
If $b\in R_{\ell}\setminus I_{\ell-1}$, then $a\in L_k(I_{k-1})$.
\end{itemize}
In fact, if {\rm (A)} and {\rm (B)} 
are shown, then the above $a$ and $b$ will satisfy
\begin{itemize}
\item[{\rm (i)}] $a\in I_k$ \ or
\item[{\rm (ii)}] $b\in I_{\ell}\,$ \ or
\item[{\rm (iii)}] $a\in L_k(I_{k-1})$ and $b\in L_{\ell}(I_{\ell-1})\setminus I_{\ell}$.
\end{itemize}
In the third case, since $a\cdot\jnd^k_{\ell}(b)\in I_k$ is satisfied by $(\ref{EqTemp})$ for $i=k$ and $j=\ell$, it follows $a\in I_k$ by $\Pcal(k)$.

\bigskip
Thus it remains to show {\rm (A)} and {\rm (B)}.

\noindent{\rm (A)}
By applying $\Pcal(k,\ell-1)$ to
\[ a\in R_k\ \ \ \text{and}\ \ \ \res^{\ell}_{\ell-1}(b)\in R_{\ell-1}, \]
we obtain $\res^{\ell}_{\ell-1}(b)\in I_{\ell-1}$, namely $b\in L_{\ell}(I_{\ell-1})$.

\noindent{\rm (B)}
When $\ell<k$, by applying $\Qcal(k-1)$, in particular $\Pcal(k-1,\ell)$ to
\[ \res^k_{k-1}(a)\in R_{k-1}\ \ \ \text{and}\ \ \ b\in R_{\ell}, \]
we obtain $\res^k_{k-1}(a)\in I_{k-1}$, namely $a\in L_k(I_{k-1})$.

When $\ell=k$, by applying $\Pcal(k,k-1)$ to
\[ b\in R_k\ \ \ \text{and}\ \ \ \res^k_{k-1}(a)\in R_{k-1}, \]
we obtain $\res^k_{k-1}(a)\in I_{k-1}$, namely $a\in L_k(I_{k-1})$.
\end{proof}
\end{proof}

By Proposition \ref{PropEP}, whether an ideal is prime or not can be checked inductively on $k$ using $\Pcal(k)$. This is applied to restrictions and extensions of prime ideals as follows.

\begin{cor}\label{CorPrimeRestr}
For any $0\le i\le k\le r$, if an ideal $\I=[I_0,\ldots,I_k]\subseteq\Omega_{H_k}$ is prime, then its restriction $\I|_{H_i}=[I_0,\ldots,I_i]$ onto $H_i$ is also prime.
\end{cor}
\begin{proof}
This immediately follows from Proposition \ref{PropEP}.
\end{proof}
\begin{cor}\label{CorPrimePrime}
For $k\ge 1$, let $\I=[I_0,\ldots,I_{k-1}]$ be an ideal of $\Omega_{H_k}$. If $\I$ is prime, then $\Lcal\I\subseteq\Omega_{H_k}$ is also prime.
\end{cor}
\begin{proof}
It suffices to show $\Pcal(k)$ is satisfied. However for $\Lcal\I=[I_0,\ldots,I_{k-1},I_k=L(I_{k-1})]$, this condition becomes trivial as follows.
\begin{itemize}
\item For any $0\le i\le k$, $a\in L_k(I_{k-1})$ and $b\in L_i(I_{i-1})\setminus I_i$,
\begin{equation}
a\cdot\jnd^k_i(b)\in I_k\ \ \Longrightarrow\ \ a\in I_k
\end{equation}
holds.
\end{itemize}
Of course this is satisfied, since $a$ belongs to $I_k=L_k(I_{k-1})$ from the first.
\end{proof}

\section{Ideals $J_{\ell,k}(x)\subseteq R_{\ell}$ }

In this section, we introduce ideals $J_{\ell,k}(x)$ of $R_{\ell}$, which will perform an essential role in determining the prime ideals of $\Omega$.

First we prepare another $\Z$-basis for $R_{\ell}$, which is more suitable for calculation.
\begin{dfn}\label{DefF}
For any $0\le i\le \ell\le r$, define $F_{\ell,i}\in R_{\ell}$ by
\[ F_{\ell,i}=X_{\ell,i}-p^{\ell-i}. \]
Obviously, each $R_{\ell}$ admits a $\Z$-basis
\[ \{ 1,F_{\ell,0},F_{\ell,1},\ldots,F_{\ell,\ell-1} \} \]
for $0\le\ell\le r$, and thus any element $\alpha\in R_{\ell}$ can be written as
\[ \alpha=m_{\ell}+\sum_{i=0}^{\ell-1}m_iF_{\ell,i} \]
for some uniquely determined $m_0,\ldots,m_{\ell}\in\Z$.
\end{dfn}
This basis behaves well with the multiplication and the structure morphisms as follows.
\begin{prop}\label{PropF}
The following holds.
\begin{enumerate}
\item For any $0\le i,j\le \ell$, we have
\[ X_{\ell,j}\cdot F_{\ell,i}=%
\begin{cases}
\ p^{\ell-j}F_{\ell,i}-p^{\ell-i}F_{\ell,j}&(i\le j),\\
\ 0&(j\le i).
\end{cases}
\]
\item For any $0\le i,k\le\ell$, we have
\[ \res^{\ell}_k(F_{\ell,i})=%
\begin{cases}
\ p^{\ell-k}F_{k,i}&(i\le k),\\
0&(k\le i).
\end{cases}
\]
In particular, we have
\[ \res^{\ell}_{\ell-1}(\alpha)=m_{\ell}+\sum_{i=0}^{\ell-2}m_ipF_{\ell-1,i} \]
for any $\alpha=m_{\ell}+\displaystyle\sum_{i=0}^{\ell-1}m_iF_{\ell,i}\in R_{\ell}$.
\item When $\ell\ge 1$, for any
\[ \alpha=n_{\ell-1}+\displaystyle\sum_{i=0}^{\ell-2}m_iF_{\ell-1,i}\in R_{\ell-1}, \]
we have
\[ \jnd^{\ell}_{\ell-1}(\alpha)=(n_{\ell-1})^p+\sum_{i=0}^{\ell-1}u_iF_{\ell,i} \]
for some $u_0,\ldots,u_{\ell-1}\in\Z$.

\item When $\ell\ge 2$, for any $0\le i\le\ell-1$, we have
\[ \ind^{\ell}_{\ell-1}(F_{\ell-1,i})=F_{\ell,i}-p^{\ell-i-1}F_{\ell,\ell-1}. \]
Moreover, $F_{\ell,\ell-1}$ is calculated as
\[ F_{\ell,\ell-1}=\ \jnd^{\ell}_{\ell-1}(F_{\ell-1,\ell-2})+\frac{(-p)^p}{p^2}\ind^{\ell}_{\ell-1}(F_{\ell-1,\ell-2}). \]
\end{enumerate}
\end{prop}
\begin{proof}
{\rm (1)} and {\rm (2)} follows immediately from Corollary \ref{CorRk}.

\noindent {\rm (3)} For any $\alpha=n_{\ell-1}+\displaystyle\sum_{i=0}^{\ell-2}m_iF_{\ell-1,i}\in R_{\ell-1}$, if we put $m_{\ell-1}=n_{\ell-1}-\displaystyle\sum_{i=0}^{\ell-2}m_ip^{\ell-i-1}$, then we have
\[ \alpha=m_{\ell-1}+\displaystyle\sum_{i=0}^{\ell-2}m_iX_{\ell-1,i} \]
and thus by Corollary \ref{Corjnd}, we obtain
\begin{eqnarray*}
\jnd^{\ell}_{\ell-1}(\alpha)&=&m_{\ell-1}+\frac{(m_{\ell-1})^p-m_{\ell-1}}{p}X_{\ell,\ell-1}\\
&+&\sum_{i=0}^{\ell-2}\frac{(\sum_{s=i}^{\ell-1}m_sp^{\ell-1-s})^p-(\sum_{s=i+1}^{\ell-1}m_sp^{\ell-1-s})^p}{p^i}X_{\ell,i}\\
&=&m_{\ell-1}+\{ (m_{\ell-1})^p-m_{\ell-1} \}+\sum_{i=0}^{\ell-2}\Set{ (\textstyle\sum_{s=i}^{\ell-1}m_sp^{\ell-1-s})^p-(\textstyle\sum_{s=i+1}^{\ell-1}m_sp^{\ell-1-s})^p }\\
&+&\frac{(m_{\ell-1})^p-m_{\ell-1}}{p}F_{\ell,\ell-1}+\sum_{i=0}^{\ell-2}\frac{(\sum_{s=i}^{\ell-1}m_sp^{\ell-1-s})^p-(\sum_{s=i+1}^{\ell-1}m_sp^{\ell-1-s})^p}{p^i}F_{\ell,i}\\
&=&(n_{\ell-1})^p+\frac{(m_{\ell-1})^p-m_{\ell-1}}{p}F_{\ell,\ell-1}\\
&+&\sum_{i=0}^{\ell-2}\frac{(n_{\ell-1}-\sum_{s=0}^{i-1}m_sp^{\ell-1-s})^p-(n_{\ell-1}-\sum_{s=0}^{i}m_sp^{\ell-1-s})^p}{p^i}F_{\ell,i}.
\end{eqnarray*}

\noindent {\rm (4)} By Remark \ref{Remind}, we have
\begin{eqnarray*}
\ind^{\ell}_{\ell-1}&&\!\!\!\!\!\!\!\!\!\!\!\!\!\!\!\!\! (F_{\ell-1,i})+p^{\ell-i-1}F_{\ell,\ell-1}\\
&=&(X_{\ell,i}-p^{\ell-i-1}X_{\ell,\ell-1})+(p^{\ell-i-1}X_{\ell,\ell-1}-p^{\ell-i})\\
&=&F_{\ell,i}.
\end{eqnarray*}
Moreover, by Remark \ref{Remind} and Corollary \ref{Corjnd}, we have
\begin{eqnarray*}
&&\!\!\!\!\!\!\!\!\!\!\!\!\!\!\!\!\!\!\!\!\!\!\jnd^{\ell}_{\ell-1}(F_{\ell-1,\ell-2})+\frac{(-p)^p}{p^2}\ind^{\ell}_{\ell-1}(F_{\ell-1,\ell-2})\\
&=&(-p+\frac{(-p)^p+p}{p}X_{\ell,\ell-1}-\frac{(-p)^p}{p^2}X_{\ell,\ell-2})\ +\ \frac{(-p)^p}{p^2}(X_{\ell,\ell-2}-pX_{\ell,\ell-1})\\
&=&X_{\ell,\ell-1}-p\ \ =\ \ F_{\ell,\ell-1}.
\end{eqnarray*}
\end{proof}

\begin{lem}\label{Lemj}
For $1\le\ell\le r$ and $n\in\Z$, we have the following.
\begin{enumerate}
\item If $n$ is not divisible by $p$, then we have
\[ n=\ \jnd^{\ell}_{\ell-1}(n)-\frac{n^{p-1}-1}{p}\ind^{\ell}_{\ell-1}(n). \]
Remark that $\frac{n^{p-1}-1}{p}$ is an integer, by assumption.
\item If $n=pu$ for some $u\in\Z$, then we have
\[ n=\ind^{\ell}_{\ell-1}(u)-uF_{\ell,\ell-1}. \]
\end{enumerate}
\end{lem}
\begin{proof}
\noindent{\rm (1)} By Remark \ref{Remind} and Corollary \ref{Corjnd}, we have
\[ \jnd^k_{k-1}(n)-\frac{n^{p-1}-1}{p}\ind^k_{k-1}(n)
=n+\frac{n^p-n}{p}X_{k,k-1}-\frac{n^{p-1}-1}{p}nX_{k,k-1}=n. \]

\noindent{\rm (2)} This follows from
\begin{eqnarray*}
\ind^{\ell}_{\ell-1}(u)-uF_{\ell,\ell-1}&=&uX_{\ell,\ell-1}-u(X_{\ell,\ell-1}-p)\\
&=&pu\ =\ n.
\end{eqnarray*}
\end{proof}

\begin{dfn}\label{DefJ}
For any $0\le\ell\le r$, $0\le k\le\ell$ and $x\in\Z$, we define an ideal $J_{\ell,k}(x)\subseteq R_{\ell}$ by
\[ J_{\ell,k}(x)=%
\begin{cases}
\ (x,F_{\ell,k},F_{\ell,k+1},\ldots,F_{\ell,\ell-1})&(k\le\ell-1),\\
\ (x)&(k=\ell).
\end{cases}
\]
\end{dfn}

\bigskip

$J_{\ell,k}(x)$ can be calculated as follows.
\begin{prop}\label{PropJ}
For any $1\le\ell\le r$, $0\le k\le\ell-1$ and $x\in\Z$, we have
\begin{eqnarray*}
J_{\ell,k}(x)&=&\Set{ x\cdot(n_{\ell}+\sum_{i=0}^{k-1}n_iF_{\ell,i})+\sum_{i=k}^{\ell-1}n_iF_{\ell,i}| \ n_0,\ldots,n_{\ell}\in\Z }\\
&=&\Set{ m_{\ell}+\sum_{i=0}^{\ell-1}m_iF_{\ell,i}| \begin{array}{l}m_0,\ldots,m_{\ell}\in\Z,\\ m_{\ell}\in\Z x,\\ m_i\in\Z x\ \ (0\le i\le k-1)\end{array}}.
\end{eqnarray*}
\end{prop}
\begin{proof}
Obviously, $J_{\ell,k}(x)$ contains any element of the form
\begin{equation}\label{alpha}
x\cdot(n_{\ell}+\sum_{i=0}^{k-1}n_iF_{\ell,i})+\sum_{i=k}^{\ell-1}n_iF_{\ell,i}.
\end{equation}
To show the converse, since any element in $J_{\ell,k}(x)$ can be written as an $R_{\ell}$-coefficient sum of
\[ x, F_{\ell,k},\ldots,F_{\ell,\ell-1}, \]
it suffices to show that any element
\begin{itemize}
\item[{\rm (i)}] $\alpha x$
\item[{\rm (ii)}] $\alpha F_{\ell,i}\ (k\le i\le\ell-1)$
\end{itemize}
can be written in the form of $(\ref{alpha})$ for any $\alpha\in R_{\ell}$.

\medskip

\noindent{\rm (i)} For any $\alpha=m_{\ell}+\displaystyle\sum_{i=0}^{\ell-1}m_iF_{\ell,i}\ (m_i\in\Z)$, we have
\[ \alpha x=x\cdot(m_{\ell}+\sum_{i=0}^{k-1}m_iF_{\ell,i})+\sum_{i=k}^{\ell-1}xm_iF_{\ell,i}. \]

\noindent{\rm (ii)} For any $\alpha=\displaystyle\sum_{j=0}^{\ell}m_jX_{\ell,j}\ (m_j\in\Z)$, we have
\begin{eqnarray*}
\alpha F_{\ell,i}&=&\sum_{j=0}^im_jX_{\ell,j}F_{\ell,i}+\sum_{j=i+1}^{\ell}m_jX_{\ell,j}F_{\ell,i}\\
&=&\sum_{j=i+1}^{\ell}m_j(p^{\ell-j}F_{\ell,i}-p^{\ell-i}F_{\ell,j})\\
&=&\big(\sum_{j=i+1}^{\ell}m_jp^{\ell-j}\big)F_{\ell,i}-\sum_{j=i+1}^{\ell}m_jp^{\ell-i}F_{\ell,j}
\end{eqnarray*}
for any $k\le i\le\ell-1$ by Proposition \ref{PropF}.
\end{proof}

\begin{prop}\label{PropLSJ}
For any $1\le\ell\le r$, $0\le k\le\ell-1$ and any $n\in\Z$, we have the following.
\begin{enumerate}
\item If $n=0$ or $n=q$ for some prime integer $q$ different from $p$, then we have
\begin{itemize}
\item[{\rm (i)}] $L(J_{\ell-1,k}(n))=J_{\ell,k}(n)$,
\item[{\rm (ii)}] $S(J_{\ell-1,k}(n))=%
\begin{cases}
\, J_{\ell,k}(n)&(k\le\ell-2),\\
\, J_{\ell,\ell}(n)=(n)&(k=\ell-1).
\end{cases}$
\end{itemize}
Moreover, when $n=q$ is a prime different from $p$, then there is no ideal between $J_{\ell,k+1}(q)\subsetneq J_{\ell,k}(q)$.

\item For $k=0$ and $n=p^{e+1}$ for some $e\in\N_{\ge0}$, we have
\begin{itemize}
\item[{\rm (i)}] $L(J_{\ell-1,0}(p^{e+1}))=J_{\ell,0}(p^{e+1})$,
\item[{\rm (ii)}] $S(J_{\ell-1,0}(p^{e+1}))=J_{\ell,0}(p^{e+2})$.
\end{itemize}
Moreover, there is no ideal between $J_{\ell,0}(p^{e+2})\subsetneq J_{\ell,0}(p^{e+1})$.
\end{enumerate}
\end{prop}
\begin{proof}
{\rm (1)} {\rm (i)} 
By definition,
\[ L(J_{\ell-1,k}(n))=\{\alpha\in R_{\ell}\mid \res^{\ell}_{\ell-1}(\alpha)\in J_{\ell-1,k}(n) \}. \]
For any $\alpha=m_{\ell}+\displaystyle\sum_{i=0}^{\ell-1}m_iF_{\ell,i}\ (m_i\in\Z)$, since we have
\begin{eqnarray*}
\res^{\ell}_{\ell-1}(\alpha)&=&m_{\ell}+\sum_{i=0}^{\ell-1}m_ipF_{\ell-1,i}\\
&=&m_{\ell}+\sum_{i=0}^{k-1}m_ipF_{\ell-1,i}+\sum_{i=k}^{\ell-1}m_ipF_{\ell-1,i},
\end{eqnarray*}
this satisfies $\res^{\ell}_{\ell-1}(\alpha)\in J_{\ell-1,k}(n)$ if and only if
\begin{eqnarray*}
&m_{\ell}\in n\Z,&\\
&m_ip\in n\Z&(0\le i\le k-1),
\end{eqnarray*}
by Proposition \ref{PropJ}. Since $n$ is $0$ or a prime different from $p$, this is equivalent to
\[ m_{\ell}\in n\Z\quad\text{and}\quad m_i\in n\Z\ (0\le i\le k-1), \]
namely, to $\alpha\in J_{\ell,k}(n)$.

\noindent{\rm (ii)}
When $k\le\ell-2$, it suffices to show
\[ J_{\ell,k}(n)\subseteq S(J_{\ell-1,k}(n)). \]
In fact, this implies
\[ J_{\ell,k}(n)\subseteq S(J_{\ell-1,k}(n))\subseteq L(J_{\ell-1,k}(n))=J_{\ell,k}(n), \]
and thus $S(J_{\ell-1,k}(n))=L(J_{\ell-1,k}(n))=J_{\ell,k}(n)$ follows.

Thus it remains to show $J_{\ell,k}(n)=(n,F_{\ell,k},\ldots,F_{\ell,\ell-1})\subseteq S(J_{\ell-1,k}(n))$. However, this immediately follows from Proposition \ref{PropF} and Lemma \ref{Lemj}, since we have
\begin{eqnarray*}
n&=&\jnd^{\ell}_{\ell-1}(n)-\frac{n^{p-1}-1}{p}\ \ \ \ (\text{for}\ n\ne0),\\
F_{\ell,\ell-1}&=&\jnd^{\ell}_{\ell-1}(F_{\ell-1,\ell-2})+\frac{(-p)^p}{p^2}\ind^{\ell}_{\ell-1}(F_{\ell-1,\ell-2}),\\
F_{\ell,i}&=&\ind^{\ell}_{\ell-1}(F_{\ell-1,i})+p^{\ell-i-1}F_{\ell,\ell-1}\ \ \ (k\le i\le \ell-1).
\end{eqnarray*}

\medskip

When $k=\ell-1$, remark that we have $J_{\ell-1,\ell-1}(q)=(q)\subseteq R_{\ell-1}$.
If $n=0$, $(0)\subseteq S(J_{\ell-1,\ell-1}(0))$ is trivial.
If $n=q$ is a prime different from $p$, then by Lemma \ref{Lemj} we have
\[ q=\jnd^{\ell}_{\ell-1}(q)-\frac{q^{p-1}-1}{p}\ind^{\ell}_{\ell-1}(q)\in S((q)),  \]
which means $(q)\subseteq S(J_{\ell-1,\ell-1}(q))$. 

Conversely, if $n=q$ is a prime different from $p$, then for any $\alpha\in R_{\ell-1}$, we have
\begin{eqnarray*}
\jnd^{\ell}_{\ell-1}(\alpha q)&=&\jnd^{\ell}_{\ell-1}(\alpha)\cdot\jnd^{\ell}_{\ell-1}(q),\\
&=&\jnd^{\ell}_{\ell-1}(\alpha)\cdot(q+\frac{q^p-q}{p}X_{\ell,\ell-1}),\\
&=&\jnd^{\ell}_{\ell-1}(\alpha)\cdot(1+\frac{q^{p-1}-1}{p}X_{\ell,\ell-1})\cdot q\ \ \in (q),\\
\ind^{\ell}_{\ell-1}(\alpha q)&=&\ind^{\ell}_{\ell-1}(\alpha)\cdot q\ \ \in (q),
\end{eqnarray*}
which imply $S(J_{\ell-1,\ell-1}(q))\subseteq (q)$. Similarly, $S(J_{\ell-1,\ell-1}(0))\subseteq (0)$ follows from $\ind^{\ell}_{\ell-1}(0)=\jnd^{\ell}_{\ell-1}(0)=0$.

\bigskip

Thus it remains to show there is no ideal between $J_{\ell,k+1}(q)\subsetneq J_{\ell,k}(q)$ for a prime $q\ne p$. 
Suppose there is an ideal
\[ J_{\ell,k+1}(q) \subsetneq I\subseteq J_{\ell,k}(q). \]
By Proposition \ref{PropJ}, $I$ should contain an element
\[ \alpha=q\beta+\sum_{i=k}^{\ell-1}n_iF_{\ell,i} \]
in $J_{\ell,k}(q)$, for some $\beta\in R_{\ell}$ and $n_i\in\Z\ \ (k\le i\le \ell-1)$, which does not belong to $J_{\ell,k+1}(q)$.
Then $I$ should contain
\[ n_kF_{\ell,k}=\alpha-\Set{q\beta+\sum_{i=k+1}^{\ell-1}n_iF_{\ell,i}}. \]
Since $\alpha$ does not belong to $J_{\ell,k+1}(q)$, it follows that $q$ does not divide $n_k$, and thus $I$ contains an element of the form
\[ n_kF_{\ell,k}\quad(n_k\in\Z, \ \text{not divisible by}\ q). \]
On the other hand, $q\in J_{\ell,k+1}(q)\subseteq I$ implies
\[ qF_{\ell,k}\in I. \]
Since $q$ and $n_k$ are coprime, it follows $F_{\ell,k}\in I$, which means $I=J_{\ell,k}(q)$.

\bigskip

\noindent {\rm (2)} {\rm (i)}
For an element $\alpha=m_{\ell}+\displaystyle\sum_{i=0}^{\ell-1}m_iF_{\ell,i}$, since we have
\[ \res^{\ell}_{\ell-1}(\alpha)=m_{\ell}+\sum_{i=0}^{\ell-2}m_ipF_{\ell-1,i} \]
by Proposition \ref{PropF}, it satisfies $\alpha\in L(J_{\ell-1,0}(p^{e+1}))$ if and only if $m_{\ell}\in p^{e+1}\Z$, namely $\alpha\in J_{\ell,0}(p^{e+1})$.

\noindent {\rm (ii)}
By Proposition \ref{PropF} and Lemma \ref{Lemj}, we have
\begin{eqnarray*}
F_{\ell,\ell-1}&=&\jnd^{\ell}_{\ell-1}(F_{\ell-1,\ell-2})+\frac{(-p)^p}{p^2}\ind^{\ell}_{\ell-1}(F_{\ell-1,\ell-2}),\\ 
F_{\ell,i}&=&\ind^{\ell}_{\ell-1}(F_{\ell-1,i})+p^{\ell-i-1}(F_{\ell,\ell-1})\ \ (0\le i\le \ell-1),\\ 
p^{e+2}&=&\ind^{\ell}_{\ell-1}(p^{e+1})-p^{e+1}F_{\ell,\ell-1},
\end{eqnarray*}
which imply $J_{\ell,0}(p^{e+2})\subseteq S(L_{\ell-1, 0}(p^{e+1}))$.

To show the converse, by Proposition \ref{PropJ}, it suffices to show any element
\[ \alpha=p^{e+1}n_{\ell-1}+\sum_{i=0}^{\ell-2}m_iF_{\ell-1,i}\ \ \ (n_{\ell-1},m_i\in\Z\ \ (0\le i\le\ell-2)) \]
in $J_{\ell-1,0}(p^{e+1})$ satisfies $\ind^{\ell}_{\ell-1}(\alpha)\in J_{\ell,0}(p^{e+2})$ and $\jnd^{\ell}_{\ell-1}(\alpha)\in J_{\ell,0}(p^{e+2})$.
However, these follow from
\begin{eqnarray*}
\ind^{\ell}_{\ell-1}(\alpha)&=&p^{e+1}n_{\ell-1}X_{\ell,\ell-1}+\sum_{i=0}^{\ell-2}m_i(F_{\ell,i}-p^{\ell-i-1}F_{\ell,\ell-1})\\
&=&p^{e+2}n_{\ell-1}+\sum_{i=0}^{\ell-2}m_iF_i+(p^{e+1}n_{\ell-1}-\sum_{i=0}^{\ell-2}p^{\ell-i-1}m_i)F_{\ell,\ell-1}
\end{eqnarray*}
and
\[ \jnd^{\ell}_{\ell-1}(\alpha)=(p^{e+1}n_{\ell-1})^p+\sum_{i=0}^{\ell-1}u_iF_{\ell,i} \]
for some $u_0,\ldots,u_{\ell-1}\in\Z$, by Proposition \ref{PropF}.

\bigskip

It remains to show there is no ideal between $J_{\ell,0}(p^{e+2})\subsetneq J_{\ell,0}(p^{e+1})$. 
Suppose there is an ideal
\[ J_{\ell,0}(p^{e+2}) \subsetneq I\subseteq J_{\ell,0}(p^{e+1}). \]
Then by Proposition \ref{PropJ}, $I$ should contain an element
\[ \alpha=p^{e+1}n_{\ell}+\sum_{i=0}^{\ell-1}m_iF_{\ell,i} \]
in $J_{\ell,0}(p^{e+1})$, for some $n_{\ell},m_i\in\Z\ \ (0\le i\le \ell-1)$, which does not belong to $J_{\ell,0}(p^{e+2})$.
Then $I$ should contain
\[ p^{e+1}n_{\ell}=\alpha-\sum_{i=0}^{\ell-1}m_iF_{\ell,i}. \]
Since $\alpha$ is not in $J_{\ell,0}(p^{e+2})$, it follows that $p$ does not divide $n_{\ell}$. 
On the other hand, we have 
\[ p^{e+2}\in J_{\ell,0}(p^{e+2})\subseteq I. \]
These imply $p^{e+1}\in I$, which means $I=J_{\ell,0}(p^{e+1})$.

\end{proof}

\section{Structure of $\mathit{Spec}\, \Omega$}
Let $0\le\ell\le r$ be any integer. For any ideal $\I=[I_0,\ldots,I_{\ell}]\subseteq\Omega_{H_{\ell}}$, define $F(\I)$ to be the ideal $F(\I)=I_0$ of $R_0$. Since $I_0$ becomes prime if $\I$ is prime, this gives a map
\[ F\colon \Spec\,\Omega_G\rightarrow\Spec\, R_0. \]
To determine $\Spec\,\Omega$, it suffices to investigate each fiber of $F$ over the points of $\Spec\, R_0$.
Since $\Spec\, R_0=\Spec\,\Z=\{0\}\cup\{ (q)\mid\, q\ \text{is a prime integer}\}$, in the rest we will study the following three cases.
\begin{itemize}
\item[{\rm (i)}] Prime ideals over $(q)$, where $q$ is a prime integer different from $p$.
\item[{\rm (ii)}] Prime ideals over $(0)$.
\item[{\rm (iii)}] Prime ideals over $(p)$.
\end{itemize}

\subsection{Prime ideals over $q\ne p$}

Let $q$ be any prime integer different from $p$.

\begin{prop}\label{PropForLSq}
We have the following.
\begin{enumerate}
\item For any $k\ge 0$, we have
\[ S^k(q)=(q)\ \ \subseteq R_k, \]
and thus $\Scal^k(q)=[(q),(q),\ldots,(q)]$.
\item For any $1\le\ell\le r$ and $0\le k\le\ell$, we have
\[ L^{\ell-k}S^k(q)=J_{\ell,k}(q)\ \ \subseteq R_{\ell}. \]
\item For any $1\le\ell\le r$ and $0\le k\le\ell-2$, we have
\[ SL^{\ell-k-1}S^k(q)=L^{\ell-k}S^k(q). \]
\end{enumerate}
\end{prop}
\begin{proof}
This follows from Proposition \ref{PropLSJ}, by an induction on $k$ and $\ell$.
\end{proof}

\begin{prop}\label{Propqqprime}
For any $0\le k\le r$,
\[ \Scal^k(q)=[(q),\ldots,(q)]\subseteq\Omega_{H_k} \]
is a prime ideal.
\end{prop}
\begin{proof}
We show this by an induction on $k$.
If $k=0$, then $(q)\subseteq\Omega_e\cong\Z$ is indeed a prime ideal.

When $k\ge 1$, suppose $\Scal^{k-1}(q)\subseteq\Omega_{H_{k-1}}$ is prime.
It suffices to show $\Scal^k(q)=[(q),\ldots,(q)]$ satisfies $\Pcal(k)$. Namely, we only have to show the following.
\begin{claim}\label{Claimqqprime}
For any $0\le i\le k$, any $a\in L_k(q)$ and any $b\in L_i(q)\setminus (q)$,
\[ a\cdot\jnd^k_i(b)\in(q)\ \ \Longrightarrow\ \ a\in(q) \]
holds.
\end{claim}
\begin{proof}
Since $L_k(q)=J_{k,k-1}(q)$, 
$a$ can be written as
\[ a=q\beta+mF_{k,k-1}\quad(\beta\in R_k,m\in\Z) \]
by Proposition \ref{PropJ}.

Similarly by Proposition \ref{PropJ}, $b\in L_i(q)=J_{i,i-1}(q)$ can be written in the form
\begin{eqnarray*}
b&=&q(n_i+\sum_{j=0}^{i-2}n_jF_{i,j})+n_{i-1}F_{i,i-1}\\
&=&q(n_i-\sum_{j=0}^{i-2}n_jp^{i-j})-pn_{i-1}+\sum_{j=0}^{i-2}n_jqX_{i,j}+n_{i-1}X_{i,i-1}
\end{eqnarray*}
for some $n_j\in\Z\ \ (0\le j\le i)$.
If we put $n_i\ppr=n_i-\displaystyle\sum_{j=0}^{i-2}n_jp^{i-j},$ then we have
\[ \jnd^k_i(b)=(qn_i\ppr-pn_{i-1})+\sum_{t=0}^{k-1}u_tX_{k,t} \]
for some $u_t\in\Z\ \ (0\le t<k)$, by Corollary \ref{Corjnd}.

Now we have
\[ a\cdot\jnd^k_i(b)=q(\beta\cdot\jnd^k_i(b))+(qn_i\ppr-pn_{i-1})mF_{k,k-1}. \]
This satisfies $a\cdot\jnd^k_i(b)\in (q)$ if and only if
\[ pn_{i-1}mF_{k,k-1}\in(q), \]
which is equivalent to $pn_{i-1}m\in q\Z$ by Proposition \ref{PropF}.
Since $b$ is not in $(q)$, $n_{i-1}$ is not divisible by $q$. Thus it follows $m\in q\Z$, which means $a\in(q)$.
\end{proof}
\end{proof}

\begin{prop}\label{Propq}
In $\Omega_{H_{\ell}}$, there are exactly $(\ell+1)$ ideals over $(q)\subseteq R_0$
\begin{equation}\label{PrimeOverq}
\Scal^{\ell}(q)\subsetneq \Lcal\Scal^{\ell-1}(q)\subsetneq\cdots\subsetneq\Lcal^{\ell-k}\Scal^k(q)\subsetneq\cdots\subsetneq\Lcal^{\ell}(q),
\end{equation}
all of which are prime.
\end{prop}
\begin{proof}
By Proposition \ref{Propqqprime} and Corollary \ref{CorPrimePrime}, these are prime.

We show the proposition by an induction on $\ell$. Suppose we have done for $\ell-1$, and take any ideal $\I=[I_0,\ldots,I_{\ell}]\subseteq\Omega_{H_{\ell}}$ over $(q)\subseteq R_0$.
By the assumption of the induction, there exists some $0\le k\le\ell-1$ satisfying
\[ \I|_{H_{\ell-1}}=\Lcal^{\ell-k-1}\Scal^k(q). \]

If $k\le\ell-2$, then $I_{\ell}$ should satisfy
\[ I_{\ell}=SL^{\ell-k-1}S^k(q)=L^{\ell-k}S^k(q)=J_{\ell,k}(q) \]
by Proposition \ref{PropForLSq}, and thus $\I=\Lcal^{\ell-k}\Scal^k(q)$.

If $k=\ell-1$, then $I_{\ell}$ should satisfy
\[ S^{\ell}(q)\subseteq I_{\ell}\subseteq LS^{\ell-1}(q). \]
Since there is no ideal between $S^{\ell}(q)=(q)$ and $LS^{\ell-1}(q)=J_{\ell,\ell-1}(q)$ by Proposition \ref{PropLSJ}, we have
\[ I_{\ell}=S^{\ell}(q)\quad\text{or}\quad LS^{\ell-1}(q), \]
namely
\[ \I=\Scal^{\ell}(q)\quad\text{or}\quad\Lcal\Scal^{\ell-1}(q). \]
\end{proof}

\subsection{Prime ideals over $p$}
\begin{prop}\label{PropSLp}
Consider an ideal in $R_{\ell}$, obtained from $(p)\subseteq R_0$ by an iterated application of $L$ and $S$
\[ S^{a_1}L^{b_1}S^{a_2}L^{b_2}\cdots S^{a_s}L^{b_s}(p)\subseteq R_{\ell} \]
for $a_i,b_i\in\N_{\ge0}$ satisfying $\displaystyle\sum_{i=1}^sa_i+\displaystyle\sum_{i=1}^sb_i=\ell$.
Then this depends only on $k=\displaystyle\sum_{i=1}^sa_i\ \ (0\le k\le \ell)$,
and is equal to $J_{\ell,0}(p^{k+1})$. Namely we have
\begin{eqnarray*}
S^{a_1}L^{b_1}S^{a_2}L^{b_2}\cdots S^{a_s}L^{b_s}(p)&=&S^kL^{\ell-k}(p)\\
&=&L^{\ell-k}S^k(p)\ =\ J_{\ell,0}(p^{k+1}).
\end{eqnarray*}
\end{prop}
\begin{proof}
By an induction on $\ell$, this immediately follows from Proposition \ref{PropLSJ}.
%
%
\end{proof}


\begin{cor}\label{CorLSp}
For any $1\le\ell\le r$, any ideal $\I=[I_0=(p),I_1,\ldots,I_{\ell}]$ of $\Omega_{H_{\ell}}$ over $(p)$ satisfies
\begin{enumerate}
\item $I_i=S(I_{i-1})$ or $I_i=L(I_{i-1})$
\item $I_i=J_{i,0}(p^{k+1})$ for some $0\le k\le i$
\end{enumerate}
for each $1\le i\le\ell$.
\end{cor}
\begin{proof}
We proceed by an induction on $\ell$. Suppose we have shown for $\ell-1$. Then, any ideal $\I=[(p),I_1,\ldots,I_{\ell}]$ should satisfy
\[ I_{\ell-1}=J_{\ell-1,0}(p^{k+1}) \]
for some $0\le k\le\ell-1$.
Then by Proposition \ref{PropLSJ}, $I_{\ell}$ should lie between $S(I_{\ell-1})=J_{\ell,0}(p^{k+2})$ and $L(I_{\ell-1})=J_{\ell,0}(p^{k+1})$, while there is no ideal between them.
Thus $I_{\ell}$ should agree with $S(I_{\ell-1})=J_{\ell,0}(p^{k+2})$ or $L(I_{\ell-1})=J_{\ell,0}(p^{k+1})$.
\end{proof}

\begin{prop}\label{Propp}
Among all the ideals of $\Omega_{H_{\ell}}$ over $(p)$ determined in Corollary \ref{CorLSp},
\[ \Lcal^{\ell}(p)=[(p),J_{1,0}(p),\ldots,J_{\ell,0}(p)] \]
is the only one prime ideal.
\end{prop}
\begin{proof}
$\Lcal^{\ell}(p)$ is prime by Corollary \ref{CorPrimePrime}.
Remark that any other ideal $\I=[I_0=(p),I_1,\ldots,I_{\ell}]$ should satisfy $I_k=S(I_{k-1})$ for some $1\le k\le\ell$. If we take the smallest such $k$, then it satisfies
\[ I_{k-1}=J_{k-1,0}(p)\quad\text{and}\quad I_k=J_{k,0}(p^2). \]
Then $\Pcal(k)$ fails for $\I$. In fact, for
\[ a=b=p\in L(I_{k-1})=J_{k,0}(p), \]
we have $ab=p^2\in J_{k,0}(p^2)$, while neither $a$ nor $b$ belong to $I_k$.
\end{proof}

\subsection{Prime ideals over $0$}

\begin{prop}\label{PropLS0}
We have the following.
\begin{enumerate}
\item For any $k\ge 0$, we have
\[ S^k(0)=(0). \]
\item For any $1\le\ell\le r$ and $0\le k\le\ell$, we have
\[ L^{\ell-k}S^k(0)=J_{\ell,k}(0)\ \subseteq R_{\ell}. \]
\item For any $1\le\ell\le r$ and $0\le k\le\ell-2$, we have
\[ SL^{\ell-k-1}S^k(0)=L^{\ell-k}S^k(0). \]
\end{enumerate}
\end{prop}
\begin{proof}
Similarly to Proposition \ref{PropForLSq}, this follows from Proposition \ref{PropLSJ} by an induction.
\end{proof}

\begin{prop}\label{Prop00prime}
For any $0\le k\le r$,
\[ \Scal^k(0)=[0,\ldots,0]=0\subseteq\Omega_{H_k} \]
is prime.
\end{prop}
\begin{proof}
This is shown in a similar way as in Proposition \ref{Propqqprime}. More generally, for any finite group $G$, it was shown that the zero ideal $0\subseteq\Omega_G$ is prime (Theorem 4.40 in \cite{N_Ideal}).
\end{proof}

\begin{lem}\label{LemIdeal0}
For any $1\le\ell\le r$ and $0\le k\le\ell-1$, any ideal $I$ satisfying
\[ J_{\ell,k+1}(0)\subseteq I\subseteq J_{\ell,k}(0) \]
is of the form
\begin{eqnarray*}
I&=&J_{\ell,k+1}(0)+(nF_{\ell,k})\\
&=&(nF_{\ell,k},F_{\ell,k+1},F_{\ell,k+2},\ldots,F_{\ell,\ell-1})
\end{eqnarray*}
for some $n\in\Z$.
\end{lem}
\begin{proof}
By Proposition \ref{PropJ}, we have
\begin{eqnarray*}
J_{\ell,k}(0)&=&\Set{\sum_{i=k}^{\ell-1}n_iF_{\ell,i}|n_i\in\Z\ \ (k\le i\le\ell-1)},\\
J_{\ell,k+1}(0)&=&\Set{\sum_{i=k+1}^{\ell-1}n_iF_{\ell,i}|n_i\in\Z\ \ (k+1\le i\le\ell-1)}.
\end{eqnarray*}
When $I\ne J_{\ell,k+1}(0)$, if we put
\[ n=\min\{ n\in\N_{>0}\mid nF_{\ell,k}\in I \}, \]
then 
we can show easily
\[ I=J_{\ell,k+1}(0)+(nF_{\ell,k}). \]
\end{proof}

\begin{prop}\label{Prop0}
In $\Omega_{H_{\ell}}$, there are exactly $\ell+1$ prime ideals
\[ 0\subsetneq\Lcal(0)\subsetneq\Lcal^2(0)\subsetneq\cdots\subsetneq\Lcal^{\ell}(0) \]
over $(0)\subseteq R_0$.

Here, $\Lcal^k(0)$ is the ideal obtained by a $k$-times application of $\Lcal$ to $0\subseteq\Omega_{H_{\ell-k}}$, for each $0\le k\le\ell$.
\end{prop}
\begin{proof}
By Corollary \ref{CorPrimePrime} and Proposition \ref{Prop00prime}, each $\Lcal^k(0)$ is prime.

We show the proposition by an induction on $\ell$. Suppose we have done for $\ell-1$, and take any prime ideal $\I=[I_0,\ldots,I_{\ell}]\subseteq\Omega_{H_{\ell}}$ over $0\subseteq R_0$.
Since $\I|_{H_{\ell-1}}$ is prime by Corollary \ref{CorPrimeRestr}, we have $\I|_{H_{\ell-1}}=\Lcal^k(0)$ on $H_{\ell-1}$ for some $0\le k\le\ell-1$, namely
\[ [I_0,\ldots,I_{\ell-1}]=[0,\ldots,0,L(0),\ldots,L^k(0)]. \]

If $k\ge 1$, then Proposition \ref{PropLS0} shows
\[ S(I_{\ell-1})=SL^k(0)=L^{k+1}(0)=L(I_{\ell-1}), \]
and thus $I_{\ell}$ should be equal to $L^{k+1}(0)$. This means $\I=\Lcal^{k+1}(0)$.

If $k=0$, then $I_{\ell}$ should satisfy
\[ 0\subseteq I_{\ell}\subseteq L_{\ell}(0)=(F_{\ell,\ell-1}). \]
By Lemma \ref{LemIdeal0}, there exists $n\in\N_{\ge0}$ satisfying $I_{\ell}=(nF_{\ell,\ell-1})$. Thus it suffices to show that $\I=[0,\ldots,0,(nF_{\ell,\ell-1})]\subseteq\Omega_{H_{\ell}}$ is not prime unless $n=0$ or $1$.

Suppose $\I$ is prime for some $n\ne0$.
Then for $a=F_{\ell,\ell-1}$ and $b=n\in R_0\setminus (0)$, since we have
\begin{eqnarray*}
a\cdot\jnd^{\ell}_0(\ell)&=&F_{\ell,\ell-1}\cdot(n+\sum_{i=0}^{\ell-1}m_iX_{\ell,i})\ \ (\text{for some}\ m_i\in\Z)\\
&=&nF_{\ell,\ell-1}\ \ \in I_{\ell},
\end{eqnarray*}
we obtain $F_{\ell,\ell-1}\in I_{\ell}$ by $\Pcal(\ell)$. This implies $I_{\ell-1}=J_{\ell,\ell-1}(0)$ and $n=1$.
\end{proof}

\subsection{Total picture}

Putting Propositions \ref{Propq}, \ref{Propp}, \ref{Prop0} together, we obtain the following.
\begin{thm}\label{ThmSpec}
Let $G$ be a cyclic $p$-group of $p$-rank $r$. The prime ideals of $\Omega_G\in\Ob(\TamG)$ are as follows.
\begin{itemize}
\item[{\rm (i)}] Over $(q)\subseteq R_0$, there are $\ell+1$ prime ideals
\[ \Scal^r(q)\subsetneq\Lcal\Scal^{r-1}(q)\subsetneq\cdots\subsetneq\Lcal^r(q). \]
\item[{\rm (ii)}] Over $(p)\subseteq R_0$, there is only one prime ideal $\Lcal^r(p)$.
\item[{\rm (iii)}] Over $0\subseteq R_0$, there are $\ell+1$ prime ideals
\[ 0\subsetneq\Lcal(0)\subsetneq\cdots\subsetneq\Lcal^r(0). \]
\end{itemize}
Thus we have
\begin{eqnarray*}
\Spec\,\Omega_G&\!\!\! =&\!\!\!\!\{ \Lcal^r(p)\}\ \cup\ \{\Lcal^i(0)\mid 0\le i\le r\}\\
&&\!\!\cup\ \,\{ \Lcal^i\Scal^{r-i}(q)\mid 0\le i\le r,\ q\ \text{is a prime different from}\ p \}.
\end{eqnarray*}
\end{thm}
For each $0\le i\le r$, there is an inclusion $\Lcal^i(0)\subsetneq\Lcal^i\Scal^{r-i}(q)$, while we have $\Lcal^j(0)\subseteq \!\!\!\!\! /\ \Lcal^i\Scal^{r-i}(q)$ for any $j>i$. Moreover, we have
\[ \Lcal^r(0)\subseteq\Lcal^r(p),\ \ \Lcal^r(p)\subseteq\!\!\!\!\! /\ \Lcal^r(q),\ \ \Lcal^r(q)\subseteq\!\!\!\!\! /\ \Lcal^r(p). \]
\begin{cor}\label{CorDim}
The longest sequence of prime ideals of $\Omega_G$ is of length $r+1$, such as
\[ 0\subsetneq\Scal^r(q)\subsetneq\Lcal\Scal^{r-1}(q)\subsetneq\cdots\subsetneq\Lcal^r(q), \]
and thus the dimension of $\Spec\,\Omega_G$, which we denote simply by $\dim\Omega_G$, is
\[ \dim\Omega_G=r+1. \]
\end{cor}

Let us draw some picture of $\Spec\,\Omega_G$, by indicating the closures of points. If $r=1$ and $G$ is a group of prime order $p$, then the picture will become as follows (\S 3.7 in \cite{N_Specp}).
\begin{ex}\label{ExPicture}
Let $G$ be the group of order $p$. Then we have
\begin{eqnarray*}
\Spec\,\Omega&\!\!\! =&\!\!\! \{\Lcal(p)\}\,\cup\,\{0\}\cup\{ \Lcal(0)\}\\
&&\!\!\cup\, \{ \Scal(q)\mid q\in\Z\ \text{is prime},\ q\ne p\}\cup\{ \Lcal(q)\mid q\in\Z\ \text{is prime},\ q\ne p\}.
\end{eqnarray*}

Inclusions are
\[
\xy
(-14,5)*+{(0)}="0";
(-7,5)*+{\subsetneq}="1";
(0,5)*+{\Lcal(0)}="2";
(7,5)*+{\subsetneq}="3";
(14,5)*+{\Lcal(p)}="4";
(-14,0)*+{\rotatebox{270}{$\subsetneq$}}="5";
(-14,-5)*+{\Scal(q)}="6";
(0,0)*+{\rotatebox{270}{$\subsetneq$}}="5";
(-7,-5)*+{\subsetneq}="9";
(0,-5)*+{\Lcal(q)}="8";
(12,-5)*+{_{(q\ne p)\ .}}="10";
\endxy
\]
Especially, the dimension is $2$.
$\Lcal(p)$ and $\Lcal(q)$'s are the closed points, and $0$ is the generic point in $\Spec\,\Omega$. If we represent the points in $\Spec\,\Omega$ by their closures, $\Spec\,\Omega$ with fibration $F$ can be depicted as in Figure $1$. 
\begin{figure}[h]
\begin{center}
\hspace{-2.6cm}\includegraphics[bb=0 0 600 600,width=4.6cm]{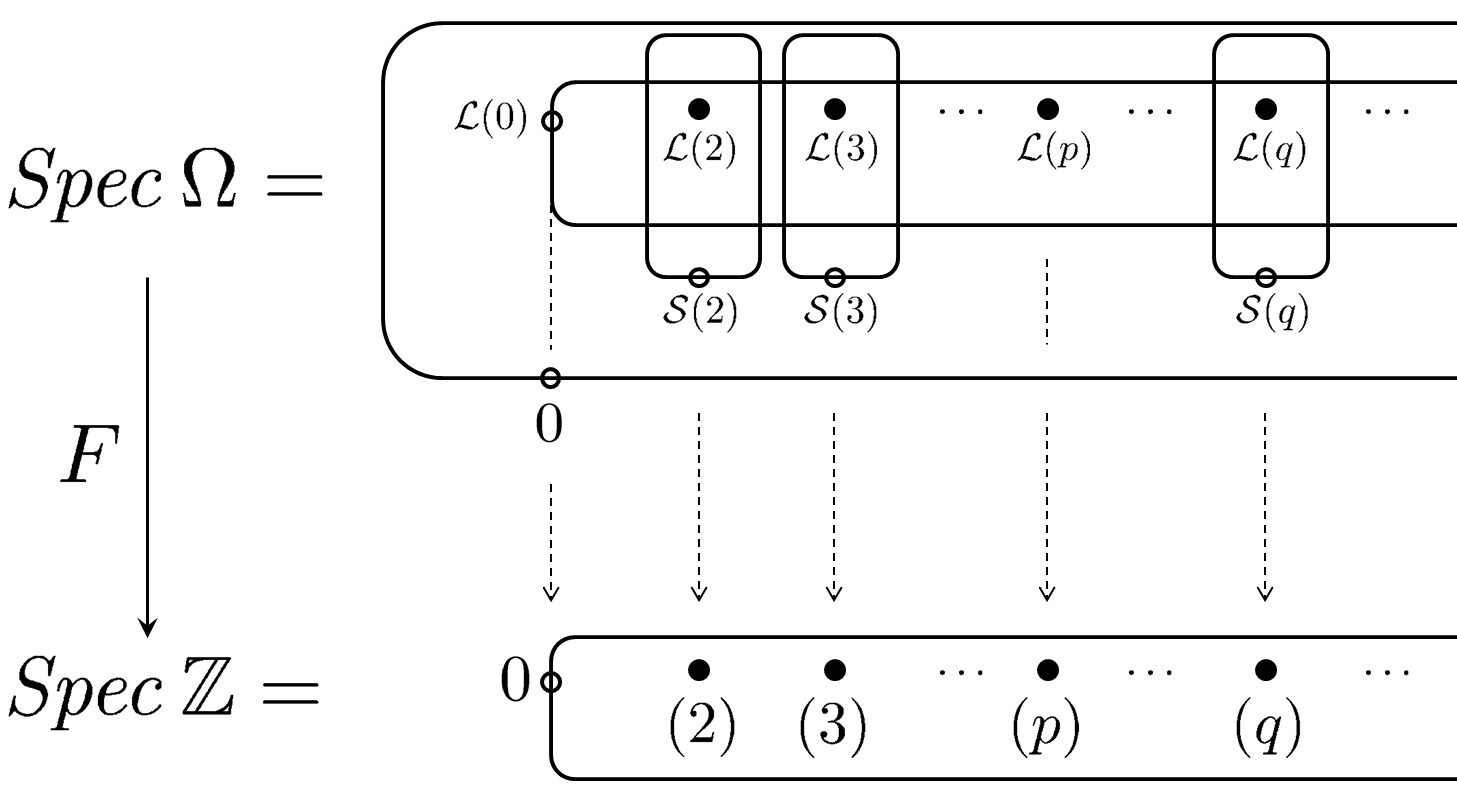}
\caption{$\Spec\,\Omega$ for $G=\Z/p\Z$}
\end{center}
\end{figure}
\end{ex}

A similar description is also possible for an arbitrary $r$. 
\begin{figure}[h]
\begin{center}
\hspace{-2.6cm}\includegraphics[bb=0 0 800 800,width=6cm]{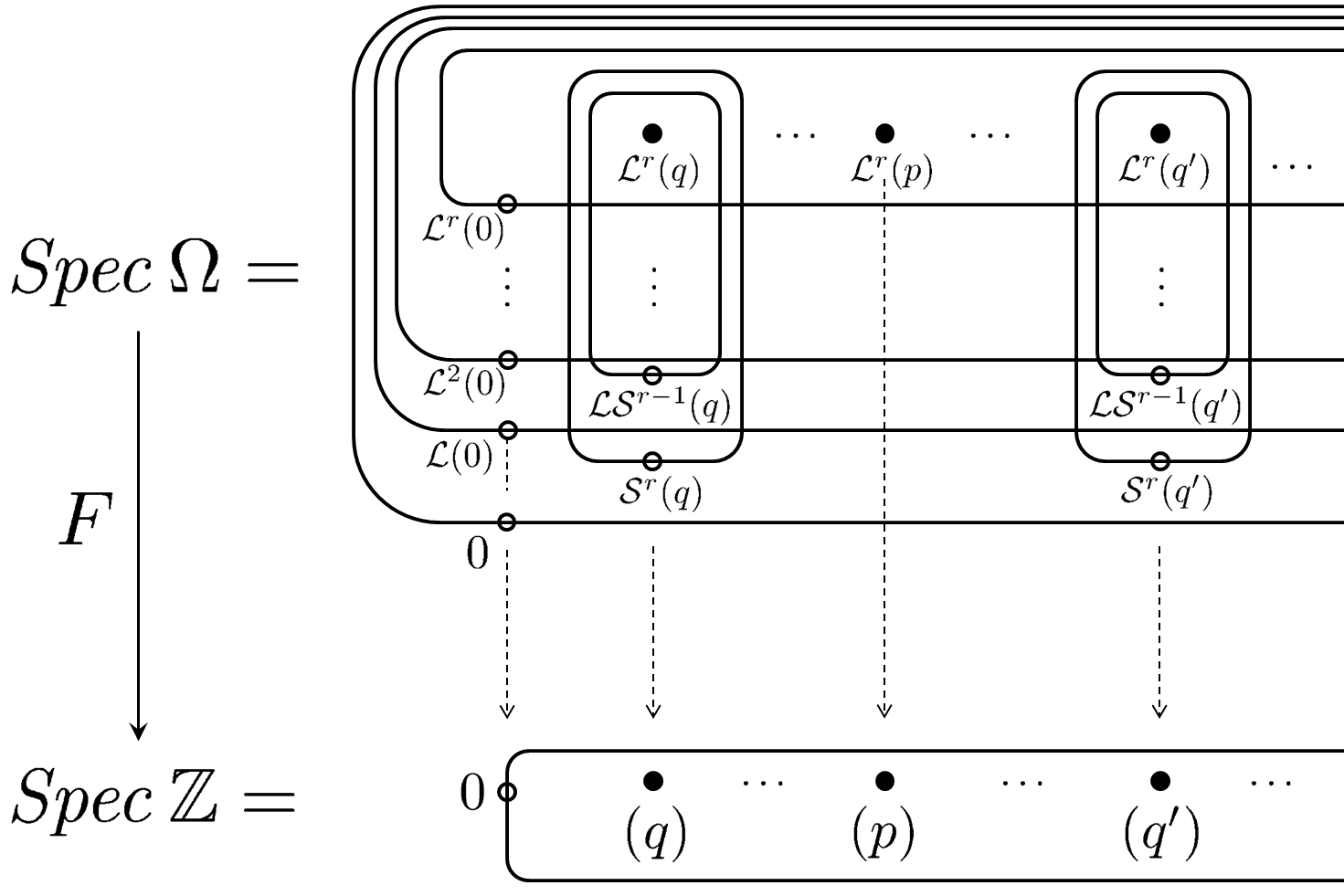}
\caption{$\Spec\,\Omega$ for $G=\Z/p^r\Z$}
\end{center}
\end{figure}
$\Lcal^r(p)$ and $\Lcal^r(q)$'s are the closed points, and $0$ is the generic point in $\Spec\,\Omega$. The picture becomes as in Figure $2$. 

\newpage
\section{Appendix: An inductive calculation of $\jnd$, when $G$ is abel}

We demonstrate how to calculate the multiplicative transfers of $\Omega$ inductively, when $G$ is an abelian group. This will show Corollary \ref{Corjnd}. A detailed investigation of an exponential map will be found in \cite{Yaman}. A more clear argument using M\"obius inversion will be found in \cite{Oda}.

Since $\jnd^L_H\colon\Omega_L(L/H)\rightarrow\Omega_L(L/L)$ and $\jnd^L_H\colon\Omega_G(G/H)\rightarrow\Omega_G(G/L)$ are essentially equal for $H\le L\le G$, we may assume $L=G$ from the first, and calculate $\jnd^G_H\colon\Omega_G(G/H)\rightarrow\Omega_G(G/G)$ for any $H\le G$.
Remark that any element in $\Omega_H(H/H)$ is of the form $\displaystyle\sum_{I\le H}m_IH/I$ for some $m_I\in\Z\ (I\le H)$, and is identified with $(\displaystyle\sum_{I\le H}m_IG/I\overset{\Pro}{\longrightarrow}G/H)\in\Omega_G(G/H)$, where $\Pro$ is the sum of the natural projections $G/I\overset{\pro^H_I}{\longrightarrow}G/H$.

Remark that $\jnd^G_H\colon\Omega_G(G/H)\rightarrow\Omega_G(G/G)$ is a polynomial map \cite{Yoshida2} as in the following. (As for polynomial maps, see \cite{Dress}.)
\begin{fact}
There exist polynomials
\[ P_K\in\Z [m_I\mid I\le H]\quad(K\le G) \]
which satisfy
\begin{equation}\label{Eq*}
\jnd^G_H(\displaystyle\sum_{I\le H}m_IG/I\overset{\Pro}{\longrightarrow}G/H)=\sum_{K\le G}P_K(\{m_I\}_{I\le H})G/K
\end{equation}
for any $\{ m_I\}_{I\le H}$.
\end{fact}
In particular if we obtain polynomials $P_K$ which satisfy $(\ref{Eq*})$ whenever $\{m_I\}_I$ satisfies $m_I\ge 0\ ({}^{\forall}I\le H)$, then  $(\ref{Eq*})$ should hold for any $\{ m_I\}_{I\le H}$ with these $P_K$'s. 
Now we calculate $P_K$, for each $K\le G$. Let $\displaystyle\sum_{I\le H}m_IH/I\in\Omega_H(H/H)$ be any element satisfying $m_I\ge0\ ({}^{\forall}I\le H)$, and denote the corresponding $G$-set $\underset{I\le H}{\coprod}(\underset{m_I}{\amalg}G/I)$ by $A$.
By the definition of $\jnd^G_H=\Omega_{\bullet}(\pro^G_H)$, the $G$-set $S=\jnd^G_H(A\overset{\Pro}{\longrightarrow}G/H)$ is given by
\begin{eqnarray*}
S&=&\{(y,s)\mid y\in G/G,\ s\in\Map((\pro^G_H)^{-1}(y),A),\ \Pro\circ s=\id_{G/H} \}\\
&=&\{ s\in \mathrm{Map}(G/H,A)\mid \Pro\circ s=\id_{G/H} \}\\
&=&\{ s\in \mathrm{Map}(G/H,A)\mid s\ \text{is a section of}\ \Pro \}
\end{eqnarray*}
(remark that $G/G$ consists of only one element). $S$ is equipped with a $G$-action
\[ {}^g\! s(x)=gs(g^{-1}x)\quad({}^{\forall}g\in G,{}^{\forall}s\in S,{}^{\forall}x\in G/H). \]

\begin{dfn}\label{DefcK}
Let $G,H,A,S$ be as above. For each $K\le G$, define $c(K)\in\Z$ by
\[ c(K)=\sharp\{ s\in S\mid G_s=K\}. \]
\end{dfn}
Then, since $G$ is abel, the number of $G$-orbits in $S$ isomorphic to $G/K$ should be equal to $\frac{c(K)}{|G:K|}$, and thus we have
\[ S=\sum_{K\le G}\frac{c(K)}{|G:K|}G/K, \]
namely, $P_K=\frac{c(K)}{|G:K|}$. Thus it remains to calculate $c(K)$.

\bigskip

Since $G$ is abel, we have the following.
\begin{rem}\label{RemGK}
If we decompose $G/K$ into $K$-orbits
\[ G/H=X_1\amalg\cdots\amalg X_{r(K)}, \]
then $r(K)=|K\backslash G/H|=|G:KH|$, and each $X_i$ is isomorphic to
\begin{equation}\label{EqSTAR}
KH/H\cong K/(K\cap H).
\end{equation}
\end{rem}

\bigskip

Fix an element $x_i\in X_i$ for each $1\le i\le r(K)$. Then, we have the following.
\begin{prop}\label{PropGK}
For any element $s\in S$, the following are equivalent.
\begin{enumerate}
\item $K\le G_s$.
\item For any $1\le i\le r(K)$ and any $k\in K$, we have
\begin{equation}\label{Eq***}
ks(x_i)=s(kx_i).
\end{equation}
\end{enumerate}
\end{prop}
\begin{proof}
By definition, $K\le G_s$ holds if and only if for any $1\le i\le r(K)$, any $x\in X_i$ and any $k\in K$,
\begin{equation}\label{Eq**}
{}^k\! s(x)=s(x)
\end{equation}
is satisfied. Since $X_i$ is $K$-transitive, there is some $k\ppr\in K$ satisfying $x=k\ppr x_i$. Then $(\ref{Eq**})$ is written as
\[ ks(k^{-1}k\ppr x_i)=s(k\ppr x_i). \]
This is easily shown to be equivalent to {\rm (2)}.
\end{proof}

\begin{cor}\label{CorGK}
Let $G,H,A,S$ be as above. For each $K\le G$, there is a bijection
\[ \{ s\in S\mid G_s\ge K\}\cong\prod_{1\le i\le r(K)}(\coprod_{I\le H}(\underset{m_I}\amalg\{ s_i\in G/I\mid \pro^H_I(s_i)=x_i,\ K_{s_i}=K_{x_i} \})). \]
\end{cor}
\begin{proof}
By Proposition \ref{PropGK}, the set of $s\in S$ satisfying $G_s\ge K$ is determined by the values $s(x_i)$ $(1\le i\le r(K))$.
In fact, to a set of elements $s_1,\ldots,s_{r(K)}\in A$, we may associate a map $s\colon G/H\rightarrow A$ satisfying $s(x_i)=s_i$, by using $(\ref{Eq***})$. This map $s$ becomes well-defined if and only if $K_{x_i}\le K_{s_i}$ is satisfied for any $i$. Moreover, we see that the following are equivalent.
\begin{itemize}
\item[{\rm (i)}] $s$ is a section of $\Pro\colon A\rightarrow G/H$, and satisfies $G_s\ge K$.
\item[{\rm (ii)}] $\Pro(s_i)=x_i\ (1\le {}^{\forall}i\le r(K))$, and $K_{s_i}=K_{x_i}$.
\end{itemize}
Thus, to give an element in $\{ s\in S\mid G_s\ge K\}$ is equivalent to give a set of elements $s_1,\ldots,s_{r(K)}\in A$ satisfying {\rm (ii)}. Namely, we have a bijection
\begin{eqnarray*}
\{ s\in S\mid G_s\ge K\}&\cong&\prod_{1\le i\le r(K)}\{ s_i\in A\mid \Pro(s_i)=x_i,\ K_{s_i}=K_{x_i} \}\\
&=&\prod_{1\le i\le r(K)}(\coprod_{I\le H}(\underset{m_I}\amalg\{ s_i\in G/I\mid \pro^H_I(s_i)=x_i,\ K_{s_i}=K_{x_i} \})).
\end{eqnarray*}
\end{proof}

\begin{cor}\label{CorGK2}
For any $K\le G$, we have
\[ \sharp\{ s\in S\mid G_s\ge K\}=(\sum_{(K\cap H)\le I\le H}\!\!\!\! m_I|H:I|\ )^{r(K)}. \]
Consequently, we can calculate $c(K)$ inductively by
\[ c(K)=(\sum_{(K\cap H)\le I\le H}\!\!\!\! m_I|H:I|\ )^{r(K)}-\sum_{K<L\le G}c(L). \]
\end{cor}
\begin{proof}
Since $G$ is abel, we have $K_{x_i}=K\cap H$ by $(\ref{EqSTAR})$.
Similarly, for each $s_i\in G/I$, we have $K_{s_i}=K\cap I$.
Since $I\le H$, we obtain
\[ K_{x_i}=K_{s_i}\ \Leftrightarrow\ K\cap H=K\cap I\ \Leftrightarrow\ K\cap H\le I. \]
As a consequence, the bijection in Corollary \ref{CorGK} is reduced to
\[ \{ s\in S\mid G_s\ge K\}\cong\prod_{1\le i\le r(K)}(\coprod_{(K\cap H)\le I\le H}(\underset{m_I}\amalg\{ s_i\in G/I\mid \pro^H_I(s_i)=x_i\})). \]

Since the fiber $(\pro^H_I)^{-1}(x_i)$ has $|H:I|$ elements, it follows
\begin{eqnarray*}
\sharp\{ s\in S\mid G_s\ge K\}&=&\prod_{1\le i\le r(K)}(\sum_{(K\cap H)\le I\le H}(m_I|H:I|))\\
&=&(\sum_{(K\cap H)\le I\le H}\!\!\!\! m_I|H:I|\ )^{r(K)}.
\end{eqnarray*}
\end{proof}

\begin{ex}\label{ExPolyCyc}
When $G=\Z/p^{\ell}\Z$, let $H_k\le G$ be the subgroup of order $p^k$ ;
\[ e=H_0<H_1<\cdots <H_{\ell}=G. \]
For $H=H_k$, we can determine $c(H_j)\ (0\le j\le\ell)$ of $\jnd^G_H(\displaystyle\sum_{0\le i\le k}m_iH/H_i)$ as follows.

By Corollary \ref{CorGK2}, for each $0\le j\le\ell$ we have
\[ c(H_j)=%
\begin{cases}
\ \ m_k& j=r,\\ \\ 
\ \ (m_k)^{p^{\ell-j}}-(m_k)^{p^{\ell-j-1}}& k\le j<r,\\ \\
\ \ (\displaystyle{\sum_{s=j}^k}m_sp^{k-s})^{p^{\ell-k}}-(\displaystyle{\sum_{s=j+1}^k}m_sp^{k-s})^{p^{\ell-k}}& 0\le j<k,
\end{cases}
\]
which leads to Corollary \ref{Corjnd}.
\end{ex}

\end{document}